\documentclass[final]{siamltex}

\usepackage[usenames,dvipsnames]{color}
\usepackage{amssymb,amsmath,amsfonts,epsfig,subfigure,epstopdf}
\usepackage{graphicx,psfrag}
\usepackage{algorithm}
\usepackage{algorithmic}
\usepackage[usenames,dvipsnames]{color}
\usepackage{stmaryrd}
\usepackage{ulem}
\usepackage{marginnote}
\def\a{\mathbf{a}}
\def\k{\mathbf{k}}
\def\p{\mathbf{p}}
\def\q{\mathbf{q}}

\def\u{\mathbf{u}}
\def\x{\mathbf{x}}

\newcommand{\LeftEqNo}{\let\veqno\@@leqno}
\usepackage{environ}
\NewEnviron{Lalign}{\tagsleft@true\begin{align}\BODY\end{align}}

\begin{document}

\title{Singular Value Decompositions for Single-Curl Operators in
  Three-Dimensional Maxwell's Equations for Complex
  Media\thanks{Version \today}}

\author{ Ruey-Lin Chern\thanks{Institute of Applied Mechanics,
    National Taiwan University, Taipei 106, Taiwan ({\tt
      chernrl@ntu.edu.tw})} \and Han-En Hsieh\thanks{Department of
    Mathematics, National Taiwan University, Taipei 106, Taiwan ({\tt
      D99221002@ntu.edu.tw}).}  \and Tsung-Ming
  Huang\thanks{Department of Mathematics, National Taiwan Normal
    University, Taipei 116, Taiwan ({\tt min@ntnu.edu.tw}).}  \and
  Wen-Wei Lin\thanks{Department of Applied Mathematics, National Chiao
    Tung University, Hsinchu 300, Taiwan ({\tt
      wwlin@math.nctu.edu.tw}).}  \and Weichung Wang\thanks{Institute
    of Applied Mathematical Sciences, National Taiwan University,
    Taipei 106, Taiwan ({\tt wwang@ntu.edu.tw}).}}

\maketitle
    
\begin{abstract}
  This article focuses on solving the generalized eigenvalue problems
  (GEP) arising in the source-free Maxwell equation with
  magnetoelectric coupling effects that models three-dimensional
  complex media. The goal is to compute the smallest positive
  eigenvalues, and the main challenge is that the coefficient matrix
  in the discrete Maxwell equation is indefinite and degenerate. To
  overcome this difficulty, we derive a singular value decomposition
  (SVD) of the discrete single-curl operator and then explicitly
  express the basis of the invariant subspace corresponding to the
  nonzero eigenvalues of the GEP. Consequently, we reduce the GEP to a
  null space free standard eigenvalue problem (NFSEP) that contains
  only the nonzero (complex) eigenvalues of the GEP and can be solved
  by the shift-and-invert Arnoldi method without being disturbed by
  the null space. Furthermore, the basis of the eigendecomposition is
  chosen carefully so that we can apply fast Fourier transformation
  (FFT)-based matrix vector multiplication to solve the embedded
  linear systems efficiently by an iterative method.  For chiral and
  pseudochiral complex media, which are of great interest in
  magnetoelectric applications, the NFSEP can be further transformed
  to a null space free generalized eigenvalue problem whose
  coefficient matrices are Hermitian and Hermitian positive definite
  (HHPD-NFGEP).  This HHPD-NFGEP can be solved by using the invert
  Lanczos method without shifting.  Furthermore, the embedded linear
  system can be solved efficiently by using the conjugate gradient
  method without preconditioning and the FFT-based matrix vector
  multiplications.  Numerical results are presented to demonstrate the
  efficiency of the proposed methods.
\end{abstract}



% ======================================================================
\section{Introduction} 
% ======================================================================

Understanding the eigenstructure of the discrete single-curl operator
$\nabla \times$ is key to developing efficient numerical simulations
for complex materials that are modeled by the Maxwell
equations. Biisotropic and bianisotropic materials are two
  important classes of complex materials \cite{WeigLakh_03}.  They
  have constantly drawn intensive studies in physical properties and
  applications.
  %%% \cite{ChengCui_06,MonzFore_05,Pendry_04}.
  For example, bianisotropic media are a special type of materials
  whose properties are characterized by the magnetoelectric as well as
  the permittivity and permeability tensors
  \cite{kong1972theorems,serdiukov2001electromagnetics}. Due to the strong
  modulation of the wave that arises from the magnetoelectric couplings,
  counterintuitive features such as negative refraction and backward
  waves may appear in bianisotropic media
  \cite{cheng2007negative,ChernChang_2013_Opt,mackay2009negative,tretyakov2007bianisotropic}.
  From the mathematical point of view, the distinctive feature
  associated with bianisotropic media is the single-curl operator
  in addition to the double-curl operator in the wave equation, which
  essentially changes the characters of the eigenwaves.

Mathematically, the propagation of electromagnetic waves in
biisotropic and bianisotropic materials is modeled by the
three-dimensional frequency domain source-free Maxwell equations
\cite{WeigLakh_03} with a set of constitutive relations.  In
particular, we have
\begin{subequations} 
  \label{eq:Maxwell_Eq}
  \begin{align}
    \nabla \times E &= \imath \omega B, \quad \nabla \cdot (\varepsilon E) = 0, \\
    \nabla \times H &= - \imath \omega D, \quad \nabla \cdot H = 0.
  \end{align}
\end{subequations}
where $\omega$ represents the frequency and $\varepsilon$ represents
the permittivity. $E$ and $H$ are the electric and magnetic fields,
respectively. Based on the Bloch theorem \cite{Kitt}, we aim to find the Bloch eigenfunctions $E$ and $H$ satisfying the
quasi-periodic conditions
\begin{align}
  E(\x + \a_{\ell}) = e^{\imath 2 \pi \k \cdot \a_{\ell}} E(\x), \quad
  H(\x + \a_{\ell}) = e^{\imath 2 \pi \k \cdot \a_{\ell}} H(\x)
  \label{eq:quasi_periodic}
\end{align}
for $\ell = 1, 2, 3$ \cite{ReedSimon_78}. Here, $\a_{1}$, $\a_{2}$,
and $\a_{3}$ are the lattice translation vectors. In this paper, we
consider the simple cubic lattice vectors $\a_{\ell} = a e_{\ell}$,
where $e_{\ell}$ is the $\ell$-th unit vector in $\mathbb{R}^{3}$ and
$a$ is a lattice constant. Note that all of the techniques developed here
can be applied to face-centered cubic lattice media.  The Bloch
wave vector in the first Brillouin zone is denoted as $2 \pi \k$
\cite{JoJoWiMe_08}.  $B$ and $D$ satisfy the constitutive relations
\begin{equation} 
  \label{eq:consti_rel_BD} 
  B =\mu H+\zeta E \mbox{ and } D =\varepsilon E+\xi H,
\end{equation}
where $\mu$ represents the permeability, and $\xi$ and $\zeta$ are
magnetoelectric parameters.  Note that $\varepsilon$, $\mu$, $\xi$, and
$\zeta$ are $3$-by-$3$ matrices in various forms for describing
different types of materials.

The Maxwell equations \eqref{eq:Maxwell_Eq} can be rewritten as
  the following quadratic eigenvalue problems (QEP), which are separate
  wave equations in terms of $E$ and $H$.
\begin{subequations} \label{eq:QEP_EH}
  \begin{align}
    \nabla \times \mu^{-1} \nabla \times E - 
    \imath  \omega \left[ \nabla \times \left( \mu^{-1} \zeta E \right) -
      \xi \mu^{-1} \nabla \times E  \right] - \omega^2 
    \left( \varepsilon - \xi \mu^{-1} \zeta \right) E &= 0;
    \label{eq:QEP_E} \\
    \nabla \times \varepsilon^{-1} \nabla \times H - \imath \omega
    \left[ \zeta \varepsilon^{-1} \nabla \times H - \nabla \times
      \left( \varepsilon^{-1} \xi H\right)\right] - \omega^2 \left(
      \mu - \zeta \varepsilon^{-1} \xi \right) H &= 0.
 \label{eq:QEP_H}
  \end{align}
\end{subequations}
In the one-dimensional case, we can apply the quasi-periodic
  conditions \eqref{eq:quasi_periodic} to \eqref{eq:QEP_EH} and then
  explicitly define the relations between $\k$ and $\omega$
  \cite{Chern_2013_PhysD,Chern_2013_Opt,ChernChang_2013_ApplPhys,ChernChang_2013_OptSoc,Lekner_96}. For
  higher dimensions, however, solving
  Eq.~\eqref{eq:quasi_periodic} efficiently remains an open question. We illustrate the
  difficulty by the following example. An explicit eigendecomposition
of the discrete double-curl operator $\nabla \times \nabla \times$ is
derived in \cite{HuangHsiehLinWang_SIMAX_13}. Applying this
eigendecomposition and assuming $\mu = 1$ and $\zeta = \xi =0$, we can
explicitly derive the invariant subspace of all nonzero eigenvalues
corresponding to the (discrete) eigenvalue problem
\eqref{eq:QEP_E}. Based on the invariant subspace, efficient numerical
methods can be developed to solve \eqref{eq:QEP_E}.  However, it is
not possible to apply this technique to solve the quadratic
eigenvalue problems \eqref{eq:QEP_EH} with $\zeta \neq 0$ and  $\xi \neq 0$ due to the following
difficulties.  (i) The eigendecomposition of the discrete double-curl
operator in \cite{HuangHsiehLinWang_SIMAX_13} cannot be applied to
solving the QEP directly because Eq. \eqref{eq:QEP_EH} contains both double-
and single-curl operators. (ii) In general, the double- and single-curl
operators in \eqref{eq:QEP_EH} cannot be diagonalized
simultaneously. Furthermore, should we find the eigendecomposition of
the single-curl operator, this decomposition cannot be applied to
solve the QEP directly because the single-curl operator terms
in \eqref{eq:QEP_EH}, e.g., $\nabla \times ( \mu^{-1} \zeta E)$ and
$\xi \mu^{-1} \nabla \times E$, are coupled with other terms such as
$\mu^{-1} \zeta$ and $\xi \mu^{-1}$.  (iii) It is difficult to find the
invariant subspace corresponding to the nonzero eigenvalues in the
quadratic eigenvalue problems.

While solving \eqref{eq:QEP_EH} is not recommended, we focus instead on the
original Maxwell equations \eqref{eq:Maxwell_Eq} and
rewrite it as a coupled generalized eigenvalue problem (GEP)
\begin{align}
  \begin{bmatrix}
    \nabla \times & 0 \\ 0 & \nabla \times
  \end{bmatrix} \begin{bmatrix} E \\ H \end{bmatrix} = \imath
  \omega \begin{bmatrix} \zeta & \mu \\ -\varepsilon & -\xi
  \end{bmatrix} \begin{bmatrix} E \\
    H \end{bmatrix}. \label{eq:conti_GEP_EH}
\end{align}
For the two-dimensional photonic band structure, the electromagnetic
transfer matrix method \cite{ChoBaiMiaRuh_97} is applied to the
coupled system, similar to \eqref{eq:conti_GEP_EH}. For the
three-dimensional case, to the best of our knowledge no method has yet been proposed to
  solve the generalized eigenvalue problem \eqref{eq:conti_GEP_EH}
  efficiently.

We make the following contributions to solve the discrete
three-dimensional generalized eigenvalue problem based on Yee's finite
difference discretization scheme \cite{Yee:66}.
\begin{itemize}
\item We first derive the singular value decomposition (SVD) of the
  discrete single-curl operator $\nabla \times $ in
  \eqref{eq:conti_GEP_EH}.

\item Using the SVD, we explore an explicit form of the basis for the
  invariant subspace corresponding to the nonzero eigenvalues of the
  GEP. Applying this basis, the GEP can be reduced to a null space
  free standard eigenvalue problem (NFSEP).  In this eigenvalue
  problem, the zero eigenvalues of the GEP are deflated so that the null
  space does not degrade the computational efficiency.

\item We show that all eigenvalues $\omega$ of the GEP are real
  provided the permittivity, permeability and magnetoelectric
  parameters satisfy particular assumptions.  These assumptions are
  applicable to a couple of important classes of complex media.
      
\item Under the same assumptions, we can reformulate the NFSEP as a
  null space free generalized eigenvalue problem $B_{r} x =
  \omega^{-1} A_{r} x$, where $B_{r}$ is Hermitian and $A_{r}$ is
  Hermitian positive definite. We demonstrate that this problem can be
  solved efficiently using the generalized Lanczos method 
  algorithmically and numerically.
\end{itemize}

This paper is outlined as follows.  In Section~\ref{sec:svd_C}, we
derive the singular value decomposition of the discrete single-curl
operator.  In Section~\ref{sec:NFM}, by applying the SVD, we derive a
null space free eigenvalue problem by deflating the zero eigenvalues
and keeping the nonzero eigenvalue unchanged. In
Section~\ref{sec:comp_app}, we discuss how to improve the solution
performance while simulating two important types of complex media.  In
Section~\ref{sec:numerical}, we demonstrate numerical results to
validate the correctness of proposed schemes and to measure the
performance of the schemes.  Finally, we present our conclusions in
Section~\ref{sec:conclude}.

Throughout this paper, we use the superscripts $\top$ and $*$ to denote the transpose
and the conjugate transpose of a matrix,
respectively. For the matrix operations, we let $\otimes$ be the
Kronecker product of two matrices.  The imaginary number $\sqrt{-1}$
is written as $\imath$, and the identity matrix of dimension $n$ is
written as $I_n$.

% ======================================================================
\section{Singular value decomposition of the discrete 
single-curl  operator} \label{sec:svd_C}
% ======================================================================

In this section, we derive an explicit expression of the SVD of the discrete
single-curl operator. Using this SVD, an efficient
null space free method to solve the target eigenvalue problem
\eqref{eq:conti_GEP_EH} is developed in Section~\ref{sec:NFM}.

We start from the derivation of the matrix representation of the
discrete single-curl operator. By using Yee's scheme \cite{Yee:66},
the discrete single-curl operators $\nabla \times E$ and $\nabla
\times H$ with $\a_{\ell} = a e_{\ell}$, $\ell = 1, 2, 3$ can be
represented in the matrix form $C E$ and $C^{\ast} H$,
respectively. Here,
\begin{align}
  {C} = \left[ \begin{array}{ccc}
      0      &   -{C}_3    &  {C}_2    \\
      {C}_3 &         0          & -{C}_1    \\
      -{C}_2 & {C}_1 & 0
    \end{array} \right] \in \mathbb{C}^{3n \times 3n}, 
  \label{eq:mtx_C}
\end{align}
with
\begin{subequations} \label{eq:mtx_C123}
  \begin{align}
    {C}_1 &= \delta_{x}^{-1} \left( {I}_{n_3} \otimes I_{n_2} \otimes {K}_{\a_1,n_1} \right)\in \mathbb{C}^{n\times n}, \\
    {C}_2 &= \delta_{y}^{-1} \left( {I}_{n_3} \otimes {K}_{{\a}_2,n_2}
      \otimes {I}_{n_1} \right) \in \mathbb{C}^{n \times n},\\
    {C}_3 &= \delta_z^{-1} \left( {K}_{{\a}_3, n_3} \otimes {I}_{n_2}
      \otimes {I}_{n_1} \right) \in \mathbb{C}^{n \times n},
  \end{align}
\end{subequations}
and %in which
\begin{align}
  {K}_{\a, m} = \left[ \begin{array}{crrr}
      -1                    &   1     &        &     \\
      & \ddots  & \ddots &     \\
      &         &  -1    &  1  \\
      e^{\imath 2\pi \k \cdot \a } & & & -1
    \end{array} \right] \in \mathbb{C}^{m \times m}. \label{eq:mtx_K_n1}
\end{align}
We use $n_1$, $n_2$, and $n_3$ to denote the numbers of grid points in the
$x$, $y$, and $z$ directions, respectively, and we define $n = n_1 n_2
n_3$.  We use $\delta_x$, $\delta_y$, and $\delta_z$ to denote the
associated mesh lengths along the $x$, $y$, and $z$ axial directions,
respectively. 

It is well known that the eigendecompositions of $C^{\ast}C$ and
$CC^{\ast}$ are closely related to the SVD of $C$. Therefore, we start the
derivation from the eigendecompositions of $C^{\ast} C$ and $C
C^{\ast}$. First, we introduce the notations to be used later.  Define
$\theta_{m,i} = \frac{\imath 2\pi i}{m}, \mbox{\ }\theta_{\a,m} =
\frac{\imath 2 \pi \k \cdot \a}{m},$
\begin{align}
% \nonumber \\
%  \theta_{m, i} + \theta_{\a, m} &\equiv \frac{\imath 2 \pi i }{m} +
%  \frac{\imath 2 \pi \k \cdot \a}{m},
  D_{\a,m} = \mbox{diag}\left( 1, e^{\theta_{\a,m}}, \cdots,  e^{(m-1)\theta_{\a,m}} \right), \label{eq:mtx_Dm} \\
  \u_{m,i} = \left[ \begin{array}{ccccc} 1 & e^{\theta_{m, i}} &
      \cdots & e^{(m-1) \theta_{m, i}}
    \end{array} \right]^{\top} \nonumber
\end{align}
for $i = 0, \ldots, m-1$ and
\begin{align}
  {U}_{m} &= \left[ \begin{array}{ccc} \u_{m,0} & \cdots & \u_{m,m-1}
    \end{array} \right]  \in \mathbb{C}^{m \times m},  \label{eq:mtx_Um} \\
  \Lambda_{\a,m} &= \mbox{diag} \left( \begin{array}{ccc}
      e^{\theta_{m, 0} + \theta_{\a, m}}-1 & \cdots & e^{\theta_{m, m-1} + \theta_{\a, m}}-1
    \end{array} \right). \nonumber %\label{eq:Lambda_am}
\end{align}
By the definition of $K_{\a,m}$ in \eqref{eq:mtx_K_n1}, it can be
verified that
\begin{align}
  K_{\a,m} \left( D_{\a,m} U_m \right) = \left( D_{\a,m} U_{m} \right)
  \Lambda_{\a,m}. \label{eq:eigdecomp_K}
\end{align}
Denote
\begin{align}
  {T} = \frac{1}{\sqrt{n}} \left( D_{\a_3,n_3}\otimes D_{\a_2,n_2}
    \otimes D_{\a_1,n_1}\right) \left({U}_{n_3} \otimes {U}_{n_2}
    \otimes {U}_{n_1} \right), \label{eq:mtx_T}
\end{align}
where $D_{\a,n_i}$ and $U_{n_i}$, $i = 1, 2, 3$, are given in
\eqref{eq:mtx_Dm} and \eqref{eq:mtx_Um}, respectively. It is
straightforward to check that $T$ is unitary.  Using
\eqref{eq:eigdecomp_K} and the property
\begin{align}
  (A_1\otimes A_2\otimes A_3)(B_1\otimes B_2\otimes B_3) =
(A_1B_1)\otimes (A_2B_2)\otimes(A_3B_3), \label{eq:prop_tensor}
\end{align}
the eigendecompositions of $C_{\ell}$ for $\ell = 1, 2, 3$ can be
obtained immediately from the following theorem.

\begin{theorem} \label{thm:SchurDecomp_Ci} $C_1$, $C_2$ and $C_3$ can
  be diagonalized by the unitary matrix $T$ in the forms
  \begin{subequations} \label{eq:eigendecomp_Cis}
    \begin{align}
      C_1 {T}& =  \delta_x^{-1} {T} \left( {I}_{n_3} \otimes {I}_{n_2} \otimes \Lambda_{{\a}_1,n_1} \right) \equiv T \Lambda_1, \label{eq:eigendecom_C1} \\
      C_2 {T}& =  \delta_y^{-1} {T} \left( {I}_{n_3} \otimes \Lambda_{{\a}_2,n_2} \otimes {I}_{n_1} \right)\equiv T\Lambda_2, \label{eq:eigendecom_C2} \\
      C_3 {T} &= \delta_z^{-1} {T} \left( \Lambda_{{\a}_3,n_3} \otimes
        {I}_{n_2} \otimes {I}_{n_1} \right) \equiv T
      \Lambda_3. \label{eq:eigendecom_C3}
    \end{align}
  \end{subequations}
\end{theorem}

We now define two intermediate matrices $\Lambda_{p}$ and
$\Lambda_{q}$ that are used in the eigendecompositions of $C^{\ast} C$
and $C C^{\ast}$.
\begin{lemma} \label{lem:nonsingular_full_rank} Let $\Lambda_1$,
  $\Lambda_2$ and $\Lambda_3$ be given in
  \eqref{eq:eigendecomp_Cis}. Define
  \begin{align}
    \Lambda_q &= \Lambda_1^{*} \Lambda_1 + \Lambda_2^{*} \Lambda_2 +
    \Lambda_3^{*} \Lambda_3, \quad \Lambda_{p} = \begin{bmatrix} \beta
      \Lambda_3 - \Lambda_2 \\ \Lambda_1 - \alpha \Lambda_3 \\ \alpha
      \Lambda_2 - \beta
      \Lambda_1 \end{bmatrix}\label{eq:mtx_Lambda_p_q}
  \end{align}
  with $\alpha, \beta \neq 0$. Assume that the vector $\k =
  (\mathrm{k}_1, \mathrm{k}_2, \mathrm{k}_3)^{\top}$ in
  \eqref{eq:quasi_periodic} is nonzero with $0 \leq \mathrm{k}_1,
  \mathrm{k}_2, \mathrm{k}_3 \leq \frac{1}{2}$.  Then, $\Lambda_{q}$ is
  positive definite, and $\Lambda_{p}$ is of full column rank, provided
  that $\alpha \delta_x \neq \beta \delta_y$ and $\delta_z \neq \beta
  \delta_y$.
\end{lemma}
\begin{proof}
  See the appendix.
\end{proof}

Using the definitions of $C$ and $\Lambda_{q}$ in \eqref{eq:mtx_C} and
\eqref{eq:mtx_Lambda_p_q}, respectively, and the eigendecompositions
of $C_{\ell}$ in Theorem~\ref{thm:SchurDecomp_Ci}, the null spaces of
$C^{\ast}C$ and $CC^{\ast}$ are derived as follows.
\begin{theorem}
  Assume $\k = (\mathrm{k}_1, \mathrm{k}_2, \mathrm{k}_3)^{\top}\neq
  0$ with $0 \leq \mathrm{k}_1, \mathrm{k}_2, \mathrm{k}_3 \leq
  \frac{1}{2}$. Define
  \begin{align}
    Q_0 = \left( I_3 \otimes T\right) \begin{bmatrix} \Lambda_1 \\
      \Lambda_2 \\ \Lambda_3
    \end{bmatrix} \Lambda_{q}^{-1/2} \equiv \left( I_3 \otimes
      T\right) \Pi_0, \quad P_0 = \left( I_3 \otimes T\right)
    \overline{\Pi_0}.
    \label{eq:mtx_Q0}
  \end{align}
  Then, $Q_0$ and $P_0$ form orthogonal bases of the null spaces of
  $C^{\ast} C$ and $C C^{\ast}$, respectively.
\end{theorem}

Next, we apply the techniques developed in
\cite{HuangHsiehLinWang_SIMAX_13} to form the orthogonal bases for the
range spaces of $C^{\ast} C$ and $CC^{\ast}$. First, by considering
the full column rank matrix $T_1 = [ \alpha T^{\top}, \beta T^{\top},
T^{\top} ]^{\top}$ with nonzero $\alpha$ and $\beta$, and taking the
orthogonal projection of $T_1$ with respect to $Q_0$ and $P_0$,
respectively, we have
\begin{subequations} \label{eq:orth_proj}
  \begin{align}
    Q_1 &= \left(I - Q_0 Q_0^{\ast} \right) T_1 \left( \Lambda_{p}^{\ast} \Lambda_{p} \Lambda_q^{-1} \right)^{-1/2} \nonumber \\
    &= \left( I_3 \otimes T\right) \begin{bmatrix}
      (\alpha \Lambda_2-\beta \Lambda_1)\Lambda_2^{\ast} - (\Lambda_1-\alpha \Lambda_3)\Lambda_3^{\ast}   \\
      (\beta \Lambda_3 -  \Lambda_2) \Lambda_3^{\ast} -(\alpha \Lambda_2 - \beta \Lambda_1) \Lambda_1^{\ast}  \\
      (\Lambda_1 - \alpha \Lambda_3) \Lambda_1^{\ast} - (\beta
      \Lambda_3 - \Lambda_2) \Lambda_2^{\ast}
    \end{bmatrix} \left( \Lambda_{p}^{\ast} \Lambda_{p} \Lambda_q^{-1} \right)^{-1/2} \nonumber \\
    &\equiv \left( I_3 \otimes T \right) \Pi_1, \label{eq:mtx_Pi_1} \\
    P_1 &= \left(I - P_0 P_0^{\ast} \right) T_1 \left(
      \Lambda_{p}^{\ast} \Lambda_{p} \Lambda_q^{-1} \right)^{-1/2} =
    \left( I_3 \otimes T \right) \overline{\Pi_1}.
  \end{align}
\end{subequations}
Then, $Q_1$ and $P_1$ are orthogonal, and $(C^{\ast} C) Q_1 = Q_1
\Lambda_{q}$ and $(C C^{\ast}) P_1 = P_1 \Lambda_{q}$.  Second, to
form the remaining part of the orthogonal basis for the range spaces of
$C^{\ast}C$ and $C C^{\ast}$, we apply the discrete curl and dual-curl
operators on $T_1$, respectively. That is, we pre-multiply $C^{\ast}$
and $C$ by $T_1$ to obtain
\begin{subequations} \label{eq:mtx_Q2}
  \begin{align}
    Q_2 &= C^{\ast} T_1 \left(\Lambda_{p}^{\ast} \Lambda_{p}
    \right)^{-1/2} = \left( I_3 \otimes T \right) \begin{bmatrix}
      \beta \Lambda_3^{\ast} -  \Lambda_2^{\ast} \\
      \Lambda_1^{\ast} - \alpha \Lambda_3^{\ast} \\
      \alpha \Lambda_2^{\ast} - \beta \Lambda_1^{\ast}
    \end{bmatrix} \left(\Lambda_{p}^{\ast} \Lambda_{p} \right)^{-1/2}
    \nonumber \\
    &\equiv \left( I_3 \otimes T \right) \Pi_2,
    \label{eq:mtx_Pi_2} \\
    P_2 &= C T_1 \left(\Lambda_{p}^{\ast} \Lambda_{p} \right)^{-1/2} =
    \left( I_3 \otimes T \right) \left( - \overline{\Pi_2} \right).
  \end{align}
\end{subequations}
It holds that $(C^{\ast} C) Q_2 = Q_2 \Lambda_{q}$ and $(C C^{\ast})
P_2 = P_2 \Lambda_{q}$.  From \eqref{eq:mtx_Q0} to \eqref{eq:mtx_Q2}, we
define
\begin{subequations}\label{eq:mtx_QdPd}
  \begin{align}
    Q &\equiv \begin{bmatrix} Q_{1} & Q_{2} & Q_{0} \end{bmatrix} =
    \left( I_3 \otimes T \right) \begin{bmatrix} \Pi_1 & \Pi_2 & \Pi_0
    \end{bmatrix}, \\
    P &\equiv \begin{bmatrix} P_2 & P_1 & P_0 \end{bmatrix} = \left(
      I_3 \otimes T \right) \begin{bmatrix} -\overline{\Pi_2} &
      \overline{\Pi_1} & \overline{\Pi_0} \,
    \end{bmatrix}.
  \end{align}
\end{subequations}
Then, the eigendecompositions of $C^{\ast} C$ and $C C^{\ast}$ can be
summarized as follows.

\begin{theorem} \label{thm:eigendecomp_A} Assume that the vector $\k =
  (\mathrm{k}_1, \mathrm{k}_2, \mathrm{k}_3)^{\top}$ in
  \eqref{eq:quasi_periodic} is nonzero with $0 \leq \mathrm{k}_1,
  \mathrm{k}_2, \mathrm{k}_3 \leq \frac{1}{2}$.  Then, $Q$ and $P$ are
  unitary, and
  \begin{align}
    C^{\ast} C = Q\, \mbox{\rm diag}\left( \Lambda_{q}, \Lambda_{q}, 0
    \right) Q^{\ast}, \quad C C^{\ast} = P\, \mbox{\rm diag}\left(
      \Lambda_{q}, \Lambda_{q}, 0\right)
    P^{\ast}. \label{eq:eigendecom_CHC}
  \end{align}
\end{theorem}
Motivated by \eqref{eq:eigendecom_CHC}, we derive the left and right
singular vector matrices $P$ and $Q$ for $C$ in the following theorem.
\begin{theorem}[Singular value decomposition of $C$] \label{thm:svd_C}
  Let $\Lambda_{q}$ and $(Q, P)$ be defined in
  \eqref{eq:mtx_Lambda_p_q} and \eqref{eq:mtx_QdPd}. Assume that the
  vector $\k = (\mathrm{k}_1, \mathrm{k}_2, \mathrm{k}_3)^{\top}$ in
  \eqref{eq:quasi_periodic} is nonzero with $0 \leq \mathrm{k}_1,
  \mathrm{k}_2, \mathrm{k}_3 \leq \frac{1}{2}$. Then, the matrix $C$ in
  \eqref{eq:mtx_C} has the SVD 
  \begin{align}
    C = P\, \mbox{\rm diag} \left( \Lambda_{q}^{1/2},
      \Lambda_{q}^{1/2}, 0 \right) Q^{\ast} = P_{r} \Sigma_{r}
    Q_{r}^{\ast},
    \label{eq:svd_C} % = \left( I_3 \otimes T\right) \overline{\Pi_{r}} \Lambda_{r} \Pi_{r}^{\ast} \left( I_3 \otimes T^{\ast} \right), \label{eq:svd_C}
  \end{align}
  where
  \begin{align*}
    P_{r} = [ P_2, \ P_1], \ Q_{r} = [ Q_1, \ Q_2], \ \Sigma_{r} &=
    \,\mbox{\rm diag}\left( \Lambda_{q}^{1/2},
      \Lambda_{q}^{1/2}\right).
  \end{align*}
\end{theorem}
\begin{proof}
  From \eqref{eq:eigendecomp_Cis} and the definition of $C$, it
  follows that
  \begin{align}
    C T_1 = - \left( I_3 \otimes T \right) \Lambda_{p}, \quad C^{\ast}
    T_1 = \left( I_3 \otimes T \right) \overline{\Lambda_{p}}, \quad C
    Q_{0} = 0, \mbox{ and } P_{0}^{\ast} C = 0. \label{eq:pf_svd_1}
  \end{align}
  Combining \eqref{eq:pf_svd_1} with \eqref{eq:orth_proj} and
  \eqref{eq:mtx_Q2}, we have
  \begin{align*}
    P_{2}^{\ast} C Q_{1} &=  \left(\Lambda_{p}^{\ast} \Lambda_{p} \right)^{-1/2} T_1^{\ast} C^{\ast} C \left(I - Q_0 Q_0^{\ast} \right) T_1 \left( \Lambda_{p}^{\ast} \Lambda_{p} \Lambda_q^{-1} \right)^{-1/2} \\
    &= \left(\Lambda_{p}^{\ast} \Lambda_{p} \right)^{-1/2} T_1^{\ast} C^{\ast} C T_1 \left( \Lambda_{p}^{\ast} \Lambda_{p} \Lambda_q^{-1} \right)^{-1/2} = \Lambda_{q}^{1/2}, \\
    P_{1}^{\ast} C Q_{1} &= \left( \Lambda_{p}^{\ast} \Lambda_{p} \Lambda_q^{-1} \right)^{-1/2} T_1^{\ast} \left( I - P_{0} P_{0}^{\ast} \right) C \left( I - Q_{0} Q_{0}^{\ast} \right) T_1 \left( \Lambda_{p}^{\ast} \Lambda_{p} \Lambda_q^{-1} \right)^{-1/2} \\
    &= \left( \Lambda_{p}^{\ast} \Lambda_{p} \Lambda_q^{-1} \right)^{-1/2} T_1^{\ast}  C T_1 \left( \Lambda_{p}^{\ast} \Lambda_{p} \Lambda_q^{-1} \right)^{-1/2} \\
    &= - \left( \Lambda_{p}^{\ast} \Lambda_{p} \Lambda_q^{-1} \right)^{-1/2} T_1^{\ast} \left( I_3 \otimes T \right) \Lambda_{p} \left( \Lambda_{p}^{\ast} \Lambda_{p} \Lambda_q^{-1} \right)^{-1/2} \\
    &= - \left( \Lambda_{p}^{\ast} \Lambda_{p} \Lambda_q^{-1} \right)^{-1/2}\left( \begin{bmatrix} \alpha I & \beta I & I \end{bmatrix} \Lambda_{p} \right) \left( \Lambda_{p}^{\ast} \Lambda_{p} \Lambda_q^{-1} \right)^{-1/2} = 0, \\
    P_{1}^{\ast} C Q_2 &= \left( \Lambda_{p}^{\ast} \Lambda_{p} \Lambda_q^{-1} \right)^{-1/2} T_1^{\ast} \left( I - P_{0} P_{0}^{\ast} \right) C C^{\ast} T_1 \left(\Lambda_{p}^{\ast} \Lambda_{p} \right)^{-1/2} \\
    &= \left( \Lambda_{p}^{\ast} \Lambda_{p} \Lambda_q^{-1} \right)^{-1/2} T_1^{\ast}  C C^{\ast} T_1 \left(\Lambda_{p}^{\ast} \Lambda_{p} \right)^{-1/2} = \Lambda_{q}^{1/2}, \\
    P_{2}^{\ast} C Q_{2} &= \left(\Lambda_{p}^{\ast} \Lambda_{p}
    \right)^{-1/2} T_1^{\ast} C^{\ast} C Q_2
    = \left(\Lambda_{p}^{\ast} \Lambda_{p} \right)^{-1/2} T_1^{\ast} Q_2 \Lambda_{q} \\
    &= \left(\Lambda_{p}^{\ast} \Lambda_{p} \right)^{-1/2} T_1^{\ast}
    C^{\ast} T_1 \left(\Lambda_{p}^{\ast} \Lambda_{p} \right)^{-1/2}
    \Lambda_{q} = 0.
  \end{align*}
  Therefore, we obtain the SVD of $C$ described in \eqref{eq:svd_C}.
  % \begin{align*}
  %   \begin{bmatrix} P_2 & P_1 & P_0 \end{bmatrix}^{\ast}
  %   C \begin{bmatrix} Q_1 & Q_2 & Q_{0} \end{bmatrix} = \mbox{diag}
  %   \left( \Lambda_{q}^{1/2},\ \Lambda_{q}^{1/2}, \ 0\right).
  % \end{align*}
\end{proof}

It is worth mentioning the specific choice of the singular vector
matrices $P$ and $Q$ in \eqref{eq:mtx_QdPd}. The choice of these
matrices is not unique because the multiplicities of the nonzero
eigenvalues of $C^{\ast}C$ are even and may be large. We have
carefully chosen the $P$ and $Q$ defined in \eqref{eq:mtx_QdPd} to
avoid the need to store these two matrices and the computations
involving $P$ and $Q$ can be performed efficiently.  We discuss
this computational advantage in Section~\ref{sec:ls_solver}.

% ======================================================================
\section{The null space free eigenvalue problem} 
\label{sec:NFM}
% ======================================================================

% Because $\nabla \times (\varepsilon E) = 0$ and $\nabla \times H =
% 0$, there are many zero eigenvalues in the discrete generalized
% eigenvalue problem \eqref{eq:conti_GEP_EH}.

We have derived the SVD of $C$ in Theorem~\ref{thm:svd_C}.  In this
section, we use the SVD to deflate the null space of the GEP obtained
by discretizing \eqref{eq:conti_GEP_EH} so that we can develop a new
solution process for the target eigenvalue problem.  The
discretization of \eqref{eq:conti_GEP_EH} based on Yee's scheme leads
to the following discrete GEP
\begin{align}
  \begin{bmatrix}
    C & 0 \\ 0 & C^{\ast}
  \end{bmatrix} \begin{bmatrix} E \\ H \end{bmatrix} = \omega \left(
    \imath \begin{bmatrix} \zeta_{d} & \mu_{d} \\ -\varepsilon_{d} &
      -\xi_{d}
    \end{bmatrix} \right) \begin{bmatrix} E \\
    H \end{bmatrix}. \label{eq:discrete_GEP_EH}
\end{align}
Here, $C$ is the discrete single-curl operator defined in
\eqref{eq:mtx_C}. The four $3n \times 3n$ complex matrices
$\zeta_{d}$, $\xi_{d}$, $\varepsilon_{d}$, $\mu_{d}$ are the discrete
counterparts of the matrices $\zeta$, $\xi$, $\varepsilon$ and $\mu$,
respectively.

From Theorem~\ref{thm:svd_C}, we can see that the GEP
\eqref{eq:discrete_GEP_EH} has $2n$ zero eigenvalues.  This null
space not only affects the convergence of iterative eigensolvers but also
increases the challenge of solving the eigenvalue problem. In this
section, we apply the SVD of $C$ in \eqref{eq:svd_C} to reduce the GEP
\eqref{eq:discrete_GEP_EH} to the null space free eigenvalue problem
\eqref{eq:NFSEP_org} equipped with the following two
advantages. (i) The dimension of the coefficient matrix is
dramatically reduced from $6n$ in \eqref{eq:discrete_GEP_EH} to $4n$
in \eqref{eq:NFSEP_org}. (ii) The two eigenvalue problems share
the same $4n$ nonzero eigenvalues.

To derive the NFSEP, we first rewrite \eqref{eq:discrete_GEP_EH} as
\begin{align}
  \mbox{\rm diag} \left( P_{r}, Q_{r} \right)\, \mbox{\rm diag} \left(
    \Sigma_{r}, \Sigma_{r} \right) \, \mbox{\rm diag} \left(
    Q_{r}^{\ast}, P_{r}^{\ast} \right) \begin{bmatrix} E \\
    H \end{bmatrix} = \omega \left( \imath \begin{bmatrix} \zeta_{d} &
      \mu_{d} \\ -\varepsilon_{d} & -\xi_{d}
    \end{bmatrix} \right) \begin{bmatrix} E \\
    H \end{bmatrix} \label{eq:svd_GEP}
\end{align}
by applying the SVD of $C$ described in \eqref{eq:svd_C}.  We can then
explicitly define the invariant subspace of \eqref{eq:discrete_GEP_EH}
in the following theorem using the matrices $P_{r}$, $Q_{r}$, and
$\Sigma_{r}$ defined in Theorem~\ref{thm:svd_C}.
\begin{theorem} \label{thm:nonzero_invariant_subsp}
  Assume the matrix
  \begin{align}
    B \equiv \imath \begin{bmatrix} \zeta_{d} & \mu_{d} \\
      -\varepsilon_{d} & -\xi_{d}
    \end{bmatrix} \label{eq:mtx_B}
  \end{align}
  is nonsingular. Then
  \begin{align*}
    \mbox{\rm span} \left\{ B^{-1} \mbox{diag} \left( P_{r}
        \Sigma_{r}^{\frac{1}{2}}, Q_{r}\Sigma_{r}^{\frac{1}{2}}
      \right) \right\}
  \end{align*}
  is an invariant subspace of \eqref{eq:discrete_GEP_EH} corresponding
  to all nonzero eigenvalues. Furthermore, it holds that
  \begin{align}
    & \left\{ \omega \left|\, \mbox{\rm diag} \left( \Sigma_{r}^{\frac{1}{2}}Q_{r}^{\ast}, \Sigma_{r}^{\frac{1}{2}} P_{r}^{\ast}\right)  B^{-1} \, \mbox{\rm diag} \left( P_{r}\Sigma_{r}^{\frac{1}{2}}, Q_{r} \Sigma_{r}^{\frac{1}{2}}\right)  y =  \omega y \right. \right\} \nonumber \\
    & = \left\{ \omega \left|\, \mbox{\rm diag} \left( C, C^{\ast}
        \right) x = \omega B x, \ \omega \neq 0
      \right. \right\}. \label{eq:nonzero_spectrum}
  \end{align}
\end{theorem}
\begin{proof}
  From Theorem~\ref{thm:svd_C}, we have
  \begin{align}
    & \mbox{\rm diag} \left( C, C^{\ast} \right) \left\{ B^{-1} \mbox{diag} \left( P_{r} \Sigma_{r}^{\frac{1}{2}}, Q_{r}\Sigma_{r}^{\frac{1}{2}} \right) \right\} \nonumber \\
    =&\ \left\{ \mbox{diag} \left( P_{r} \Sigma_{r}Q_{r}^{\ast}, Q_{r}\Sigma_{r} P_{r}^{\ast}\right)  \right\} \left\{ B^{-1} \mbox{diag} \left( P_{r} \Sigma_{r}^{\frac{1}{2}}, Q_{r}\Sigma_{r}^{\frac{1}{2}} \right) \right\} \nonumber \\
    =&\ B \left\{ B^{-1} \mbox{diag}\left(
        P_{r}\Sigma_{r}^{\frac{1}{2}}, Q_{r}\Sigma_{r}^{\frac{1}{2}}
      \right) \right\} \left\{ \mbox{diag} \left(
        \Sigma_{r}^{\frac{1}{2}}Q_{r}^{\ast}, \Sigma_{r}^{\frac{1}{2}}
        P_{r}^{\ast} \right) B^{-1} \mbox{diag} \left(
        P_{r}\Sigma_{r}^{\frac{1}{2}}, Q_{r} \Sigma_{r}^{\frac{1}{2}}
      \right) \right\}. \label{eq:pf_invariant_1}
  \end{align}
  From \eqref{eq:svd_GEP}, \eqref{eq:pf_invariant_1}, and the fact
  that $\mbox{diag} \left( \Sigma_{r}^{\frac{1}{2}}Q_{r}^{\ast},
    \Sigma_{r}^{\frac{1}{2}} P_{r}^{\ast} \right) B^{-1} \mbox{diag}
  \left( P_{r}\Sigma_{r}^{\frac{1}{2}}, Q_{r} \Sigma_{r}^{\frac{1}{2}}
  \right) \in \mathbb{C}^{4n \times 4n}$ is nonsingular, we can see
  that span$\{ B^{-1} \mbox{diag} \left( P_{r}
    \Sigma_{r}^{\frac{1}{2}}, Q_{r}\Sigma_{r}^{\frac{1}{2}} \right)
  \}$ is an invariant subspace of the GEP \eqref{eq:discrete_GEP_EH}
  corresponding to all nonzero eigenvalues, and therefore the result in
  \eqref{eq:nonzero_spectrum} holds.
\end{proof}

From Theorem~\ref{thm:nonzero_invariant_subsp}, the null space free
eigenvalue problem is derived straightforwardly in the following 
theorem.
\begin{theorem}[The null space free eigenvalue
  problem] \label{thm:nfep} For any nonsingular $B$ defined in
  \eqref{eq:mtx_B}, the GEP \eqref{eq:discrete_GEP_EH} can be reduced
  to the following null space free eigenvalue problem
\begin{align}
  \mbox{\rm diag} \left( \Sigma_{r}^{\frac{1}{2}}Q_{r}^{\ast},
    \Sigma_{r}^{\frac{1}{2}} P_{r}^{\ast} \right) B^{-1} \,\mbox{\rm
    diag} \left( P_{r}\Sigma_{r}^{\frac{1}{2}}, Q_{r}
    \Sigma_{r}^{\frac{1}{2}} \right) y = \omega y. \label{eq:NFSEP_org}
\end{align}
Furthermore, the GEP \eqref{eq:discrete_GEP_EH} and the NFSEP
\eqref{eq:NFSEP_org} have the same nonzero eigenvalues.
\end{theorem}

We have reduced the GEP \eqref{eq:discrete_GEP_EH} to the NFSEP
\eqref{eq:NFSEP_org}, and the NFSEP can be solved by the iterative
eigensolvers without being disturbed by the null space. Further
computational considerations are discussed in the next section.
Finally, we note how the $\mathbb{C}^{3n \times 3n}$ matrices
$\zeta_{d}$, $\xi_{d}$, $\varepsilon_{d}$, and $\mu_{d}$ are determined.  In
Yee's scheme, $E$ and $H$ are evaluated at the edge centers and the
face centers, respectively. However, $CE$ and $C^{\ast} H$ are evaluated at the face
centers and the edge centers, respectively. To match these evaluation points, we can
average the corresponding discrete entry values of $\zeta$, $\xi$,
$\varepsilon$ and $\mu$ on the neighbor grid points to form the
matrices $\zeta_{d}$, $\xi_{d}$, $\varepsilon_{d}$, $\mu_{d}$ in
\eqref{eq:discrete_GEP_EH}.  Consequently, $\zeta_{d} E + \mu_{d} H$
and $\varepsilon_{d} E + \xi_{d} H$ are evaluated at the face centers and the
edge centers, respectively.

% ======================================================================
\section{Computational and application considerations}
\label{sec:comp_app}
% ======================================================================

The NFSEP defined in \eqref{eq:NFSEP_org} can actually be applied to
various complex media settings as long as the corresponding matrix $B$
is nonsingular. Such media include general and Tellegen biisotropic
media \cite{WangZhouKoscKafeSouk_09}, lossless and reciprocal
bianisotropic media \cite{SerdSemcTretSihv_01}, and general
bianisotropic media \cite{Kamentskii_98}. To solve the NFSEP,
shift-and-invert type iterative eigensolvers (e.g., Arnoldi method,
Jacobi-Davidson method) can be applied to compute the desired
eigenpairs of \eqref{eq:discrete_GEP_EH} from \eqref{eq:NFSEP_org}
without being affected by zero eigenvalues.  Despite the wide
applications on complex media, the process for solving the NFSEP
\eqref{eq:NFSEP_org} can be can further accelerated under the mild
assumption described in Section~\ref{sec:hpd}.  It is worth mentioning
that two important types of media, i.e., chiral media
\cite{Chern_2013_Opt,ChernChang_2013_ApplPhys,LiuZhang_11,SihvViitLindTret_book_94,WangZhouKoscKafeSouk_09,ZhaoKoscSouk_10}
and pseudochiral media
\cite{Chern_2013_PhysD,ChernChang_2013_Opt,ChernChang_2013_OptSoc,Kamentskii_98,TretSihvSochSimo_98},
satisfy this assumption and can thus be solved by the accelerated
eigensolvers. See Sections~\ref{sec:ls_solver} and
\ref{sec:app_remarks} for more details and some computational remarks.

% ======================================================================
\subsection{Sufficient conditions for Hermitian and Hermitian positive
  definite generalized eigenvalue problems}
\label{sec:hpd}
% ======================================================================

The coefficient matrix in the NFSEP \eqref{eq:NFSEP_org} is in a
general form, and the NFSEP can therefore be solved using, for example, the Arnoldi
method. However, under an assumption, we can rewrite the NFSEP
\eqref{eq:NFSEP_org} as a generalized eigenvalue problem with
Hermitian and Hermitian positive definite coefficient matrices.  We
can then take advantage of the matrix structure to accelerate the
solution process by solving this rewritten eigenvalue problem via the
invert Lanczos method and the associated conjugate gradient linear
system solver.

The acceleration scheme is motivated from the following observations
regarding the matrix $B$ in \eqref{eq:mtx_B}.  If $\mu_{d}$ is
nonsingular and we let $\Phi$ represent the matrix $\varepsilon_{d} - \xi_{d}
\mu_{d}^{-1} \zeta_{d}$, we have
\begin{align*}
  B = \imath \begin{bmatrix} 0 & \mu_{d} \\ -\Phi & -\xi_{d}
  \end{bmatrix} \begin{bmatrix} I_{3n} & 0 \\ \mu_{d}^{-1} \zeta_{d}
    & I_{3n}
  \end{bmatrix}.
\end{align*}
Furthermore, if $\Phi$ is nonsingular, we have
\begin{align}
  B^{-1} = - \imath \begin{bmatrix} I_{3n} & 0 \\
    -\mu_{d}^{-1}\zeta_{d} & I_{3n}
  \end{bmatrix} \begin{bmatrix} -\Phi^{-1} \xi_{d}\mu_{d}^{-1} &
    -\Phi^{-1} \\ \mu_{d}^{-1} & 0
  \end{bmatrix}. \label{eq:B_inv}
\end{align}
In other words, the properties of $B^{-1}$ are closely related to
$\mu_{d}$, $\Phi$, $\xi_{d}$ and $\zeta_{d}$.  In particular, we
consider the assumption
\begin{equation}
  \label{eq:assump}
  \mu_{d} \succ 0, \Phi \equiv \varepsilon_{d} - \xi_{d} \mu_{d}^{-1} \zeta_{d}
  \succ 0, \mbox{ and } \xi_{d}^{\ast} = \zeta_{d}.
\end{equation}
The notations $\mu_{d} \succ 0$ and $\Phi \succ 0$ suggest that
$\mu_{d}$ and $\Phi$ are Hermitian positive definite.  Under this
assumption, we show that the GEP \eqref{eq:discrete_GEP_EH} can be
transformed to a standard Hermitian eigenvalue problem (so that all of the
eigenvalues are real) in Theorem~\ref{lem:skew-symm_A}.  We then
rewrite the NFSEP \eqref{eq:NFSEP_org} in the new form
\eqref{eq:NFSEP_spd_1} in Theorem~\ref{thm:NFSEP_real_ew}. Consequently,
the corresponding coefficient matrix $A_{r}$ to be defined in
\eqref{eq:factor_mtx} is Hermitian and positive definite. We can then
use the Lanczos method, which consumes less storage and computation than
Arnoldi-type methods, to solve \eqref{eq:NFSEP_spd_1}.

We begin the derivation from the following theorem.
\begin{theorem} \label{lem:skew-symm_A} Under
  Assumption~\eqref{eq:assump}, all eigenvalues $\omega$ of the GEP
  \eqref{eq:discrete_GEP_EH} are real.
\end{theorem}
\begin{proof}
  Let
  \begin{align}
    \begin{bmatrix}
      E_{\zeta} \\ H_{\zeta}
    \end{bmatrix} = \begin{bmatrix} I_{3n} & 0 \\
      \mu_{d}^{-1}\zeta_{d} & I_{3n}
    \end{bmatrix} \begin{bmatrix} E \\ H
    \end{bmatrix}. \label{eq:sub_1}
  \end{align}
  Substituting \eqref{eq:sub_1} into \eqref{eq:discrete_GEP_EH} and
  pre-multiplying \eqref{eq:discrete_GEP_EH} by $\begin{bmatrix}
    I_{3n} & 0 \\ \xi_{d}\mu_{d}^{-1} & I_{3n} \end{bmatrix}$, it
  holds that
  \begin{align}
    \begin{bmatrix}
      C & 0 \\ \xi_{d} \mu_{d}^{-1}C-C^{\ast} \mu_{d}^{-1}\zeta_{d} &
      C^{\ast}
    \end{bmatrix} \begin{bmatrix} E_{\zeta} \\ H_{\zeta}
    \end{bmatrix} = \imath \omega \begin{bmatrix} 0 & \mu_{d} \\ -\Phi
      & 0
    \end{bmatrix} \begin{bmatrix} E_{\zeta} \\ H_{\zeta}
    \end{bmatrix}, \label{eq:discrete_GEP_EH_2}
  \end{align}
  where $\Phi$ is defined in Assumption~\eqref{eq:assump}.
  By the assumptions that $\mu_{d} \succ 0$ and $\Phi \succ 0$, we then let
  \begin{align}
    \mu_{d} = \mu_{c} \mu_{c}^{\ast}, \quad \Phi = \Phi_{c}
    \Phi_{c}^{\ast} \label{eq:Phi_12}
  \end{align}
  be the Cholesky decompositions of $\mu_{d}$ and $\Phi$, 
respectively.  Define
  \begin{align}
    \begin{bmatrix}
      \widetilde{E} \\ \widetilde{H}
    \end{bmatrix} = \begin{bmatrix} \Phi_{c}^{\ast} & 0 \\ 0 &
      \mu_{c}^{\ast}
    \end{bmatrix} \begin{bmatrix} E_{\zeta} \\ H_{\zeta}
    \end{bmatrix}. \label{eq:sub_2}
  \end{align}
  Substituting \eqref{eq:sub_2} into \eqref{eq:discrete_GEP_EH_2} and
  pre-multiplying \eqref{eq:discrete_GEP_EH_2} by $\begin{bmatrix} 0 &
    -\Phi^{-1}_{c} \\ \mu_{c}^{-1} & 0 \end{bmatrix}$, we have 
  \begin{equation}
    \label{eq:HSEP}
    A x = \omega x,
  \end{equation}
  where $x = [ \widetilde{E}^{\top}, \ \widetilde{H}^{\top} ]^{\top}$
  and
  \begin{align}
    A = (-\imath) \begin{bmatrix}
      - \Phi^{-1}_{c} \left( \xi_{d} \mu_{d}^{-1} C-C^{\ast} \mu_{d}^{-1} \zeta_{d}  \right) \left(\Phi_{c}^{\ast} \right)^{-1} & - \Phi^{-1}_{c} C^{\ast} \left(\mu_{c}^{\ast} \right)^{-1} \\
      \mu_{c}^{-1} C\left(\Phi_{c}^{\ast} \right)^{-1} & 0
    \end{bmatrix}. \label{eq:mtx_A}
  \end{align}
  Because $A$ is Hermitian, all eigenvalues $\omega$ of the GEP
  \eqref{eq:discrete_GEP_EH} are real.
\end{proof}

We have shown that all eigenvalues of \eqref{eq:discrete_GEP_EH} are
real under Assumption~\eqref{eq:assump}. However, the coefficient
matrix of the NFSEP in \eqref{eq:NFSEP_org} is not Hermitian. We
reformulate the NFSEP \eqref{eq:NFSEP_org} in the following theorem to
obtain a Hermitian and Hermitian positive definite generalized
eigenvalue problem (HHPD-GEP).
\begin{theorem} \label{thm:NFSEP_real_ew} Under
  Assumption~\eqref{eq:assump}, the GEP \eqref{eq:discrete_GEP_EH} can
  be reduced to a null space free generalized eigenvalue problem
  \begin{align}
    \left( \imath \begin{bmatrix} 0 & \Sigma_{r}^{-1} \\
        -\Sigma_{r}^{-1} & 0 \end{bmatrix} \right) y_{r} = \omega^{-1}
    A_{r} y_{r}, \label{eq:NFSEP_spd_1}
  \end{align}
  where
  \begin{align}
    A_{r} \equiv \,\mbox{\rm diag} \left( P_{r}^{\ast}, Q_{r}^{\ast}
    \right) \begin{bmatrix} \mu_{d}^{-1}\zeta_{d} & -I_{3n} \\ I_{3n}
      & 0
    \end{bmatrix} \begin{bmatrix} \Phi^{-1} & 0 \\ 0 & \mu_{d}^{-1}
    \end{bmatrix} \begin{bmatrix} \xi_{d}\mu_{d}^{-1} & I_{3n} \\
      -I_{3n} & 0
    \end{bmatrix} \,\mbox{\rm diag} \left( P_{r}, Q_{r}
    \right) \label{eq:factor_mtx}
  \end{align}
  is Hermitian and positive definite.
\end{theorem}
\begin{proof}
  Let
  \begin{align*}
    y_{r} = \mbox{diag} \left( \Sigma_{r}^{1/2}, \Sigma_{r}^{1/2}
    \right) y.
  \end{align*}
  Rewrite \eqref{eq:NFSEP_org} as
  \begin{align*}
    \mbox{diag} \left( Q_{r}^{\ast}, P_{r}^{\ast} \right) B^{-1}
    \mbox{diag} \left( P_{r}, Q_{r} \right) y_{r} = \omega\,
    \mbox{diag} \left( \Sigma_{r}^{-1}, \Sigma_{r}^{-1} \right) y_{r},
  \end{align*}
  which is equivalent to
  \begin{align}
    \imath \, \mbox{diag} \left( P_{r}^{\ast}, Q_{r}^{\ast}
    \right) \begin{bmatrix} 0 & I_{3n} \\ -I_{3n} & 0 \end{bmatrix}
    B^{-1} \mbox{diag} \left( P_{r}, Q_{r} \right) y_{r} = \omega
    \left( \imath \begin{bmatrix} 0 & \Sigma_{r}^{-1} \\
        -\Sigma_{r}^{-1} & 0 \end{bmatrix} \right)
    y_{r}. \label{eq:NFSEP_spd}
  \end{align}
  From \eqref{eq:B_inv}, it holds that
  \begin{align*}
    \imath \begin{bmatrix} 0 & I_{3n} \\ -I_{3n} & 0 \end{bmatrix}
    B^{-1} &= \begin{bmatrix} \mu_{d}^{-1} \zeta_{d} & -I_{3n} \\
      I_{3n} & 0
    \end{bmatrix} \begin{bmatrix} \Phi^{-1} \xi_{d}\mu_{d}^{-1} &
      \Phi^{-1} \\ -\mu_{d}^{-1} & 0
    \end{bmatrix} \\
    &= \begin{bmatrix} \mu_{d}^{-1}\zeta_{d} & -I_{3n} \\ I_{3n} & 0
    \end{bmatrix} \begin{bmatrix} \Phi^{-1} & 0 \\ 0 & \mu_{d}^{-1}
    \end{bmatrix} \begin{bmatrix} \xi_{d}\mu_{d}^{-1} & I_{3n} \\
      -I_{3n} & 0
    \end{bmatrix},
  \end{align*}
  which is Hermitian and positive definite if
  Assumption~\eqref{eq:assump} holds. That is, the coefficient matrix
  on the left hand side of \eqref{eq:NFSEP_spd} is equal to $A_{r}$.
  Therefore, the equation \eqref{eq:NFSEP_spd} can be rewritten as
  \eqref{eq:NFSEP_spd_1}.
\end{proof}

We have now asserted the sufficient conditions that lead to the Hermitian
and Hermitian positive definite null space free generalized eigenvalue
problem (HHPD-NFGEP) \eqref{eq:NFSEP_spd_1}. Next, we discuss some
considerations in applying and solving the HHPD-NFGEP.

% ======================================================================
\subsection{The eigenvalue and associated linear system solvers}
\label{sec:ls_solver}
% ======================================================================

In the HHPD-NFGEP \eqref{eq:NFSEP_spd_1}, the coefficient matrix
$A_{r}$ is Hermitian and positive definite. We can use the generalized
Lanczos method to solve \eqref{eq:NFSEP_spd_1} and obtain the
smallest positive eigenvalues that are of interest in complex
media. In each step of the generalized Lanczos method, we must
solve the linear systems
\begin{align}
  A_{r} u \equiv \begin{bmatrix} P_{r}^{\ast} & \\ &
    Q_{r}^{\ast} \end{bmatrix} \begin{bmatrix} \zeta_{d} & -I_{3n} \\
    I_{3n} & 0
  \end{bmatrix} \begin{bmatrix} \Phi^{-1} & 0 \\ 0 & I_{3n}
  \end{bmatrix} \begin{bmatrix} \zeta_{d}^{\ast} & I_{3n} \\ -I_{3n} &
    0
  \end{bmatrix} \begin{bmatrix} P_{r} & \\ & Q_{r} \end{bmatrix} u =
  b \label{eq:spd_LS}
\end{align}
for a given vector $b$.  Because $A_{r}$ is Hermitian positive
definite, the linear system \eqref{eq:spd_LS} can be solved efficiently by using the
conjugate gradient method.  Furthermore, the matrix-vector
multiplications of the forms $(T^{\ast} \p, T\q)$ for computing
$(P_{r}^{\ast}\hat{\p}_1, Q_{r}^{\ast} \hat{\p}_2)$ and
$(P_{r}\hat{\q}_1, Q_{r}\hat{\q}_2)$, which are the most costly parts
of solving \eqref{eq:spd_LS}, can be computed efficiently by the
three-dimensional FFT because of the periodicity of $T$, as shown in
\eqref{eq:mtx_T}.

% ======================================================================
\subsection{Application remarks}
\label{sec:app_remarks}
% ======================================================================

Intensive research has been conducted on chiral and
pseudochiral media. In this section, we assert that the corresponding
magnetoelectric matrices satisfy the assumptions given in \eqref{eq:assump}, and we can therefore solve the HHPD-NFGEP using the Lanczos
method. This solution procedure can thus act as a useful numerical
tool for simulating three-dimensional chiral and pseudochiral
media.

First, we introduce the magnetoelectric matrices in chiral and
pseudochiral media.  For chiral media (also called Pasteur or
reciprocal chiral media), the associated magnetoelectric matrices
$\zeta_{d}$ and $\xi_{d}$ in \eqref{eq:discrete_GEP_EH} are of the
forms
\begin{align}
  \xi_{d} = \imath \gamma \tilde{I}_{3n} \mbox{ and } \zeta_{d} =
  -\imath \gamma \tilde{I}_{3n}.
  \label{eq:chiral_para_mtx}
\end{align}
Here, $\gamma$ is the chirality parameter, and $\tilde{I}_{m} \in
\mathbb{R}^{m \times m}$ is a diagonal matrix whose entries are
equal to $0$ (outside of the medium) or $1$ (within the medium)
depending on the corresponding grid point locations.  For 
pseudochiral media, the associated matrices are
\begin{align}
 \xi_{d} = \begin{bmatrix} 0 & 0 & \imath \gamma \tilde{I}_{n} \\
    0 & 0 & 0 \\ \imath \gamma \tilde{I}_{n} & 0 & 0
  \end{bmatrix}
\mbox{ and }
  \zeta_{d} = \begin{bmatrix} 0 & 0 & -\imath \gamma \tilde{I}_{n} \\
    0 & 0 & 0 \\ -\imath \gamma \tilde{I}_{n} & 0 & 0
  \end{bmatrix}.
\label{eq:pseudo_chiral_para_mtx}
\end{align}

Now, we analyze these magnetoelectric matrices.  For chiral and
pseudochiral media, $\varepsilon$ are of the form of $3$-by-$3$
diagonal block matrix that $\varepsilon = \mbox{diag}(\epsilon,
\epsilon, \epsilon)$. Here, $\epsilon$ is a piecewise constant
function that is equal to $\varepsilon_{i}$ and $\varepsilon_{o}$ inside and outside the medium, respectively. Thus, the
associated matrix $\varepsilon_{d}$ in \eqref{eq:discrete_GEP_EH} is a
diagonal matrix with $\varepsilon_{o}$ or $\varepsilon_{i}$ on
diagonal entries. On the other hand, the permeability $\mu$ is usually
taken as $I_3$, so the matrix $\mu_{d}$ in
\eqref{eq:discrete_GEP_EH} is equal to an identity. Combining the
diagonal matrices $\varepsilon_{d}$ and $\mu_{d}$ with $(\xi_{d}, \zeta_{d})$
in \eqref{eq:chiral_para_mtx} and \eqref{eq:pseudo_chiral_para_mtx},
we have that $\Phi = \varepsilon_{d} - \xi_{d} \mu_{d}^{-1}\zeta_{d} = \varepsilon_{d} -
\xi_{d}\zeta_{d}$ is a diagonal matrix with the entries 
$\varepsilon_{o}$, $\varepsilon_{i}$, or $\varepsilon_{i} -
\gamma^2$. Because $\varepsilon_{o}$ and $\varepsilon_{i}$ are positive,
$\Phi$ is a positive diagonal matrix, provided $\gamma \in
(0,\sqrt{\varepsilon_{i}})$.

% ======================================================================
\subsection{A short summary}
\label{sec:comp_app_summary}
% ======================================================================

In Table~\ref{tab:summary}, we summarize all of the eigenvalue
problems and eigensolver strategies that have been discussed in
the previous sections. From the algorithmic viewpoint, we propose and
outline the Null Space Free method (NSF) in
Algorithm~\ref{alg:NullSpaceFree}.
 
% ======================================================================
\begin{table}
  \centering
\begin{tabular}[t]{lllll}
\hline
(a) & GEP \eqref{eq:conti_GEP_EH} & NFSEP \eqref{eq:NFSEP_org} & 
HSEP \eqref{eq:HSEP} & HHPD-NFGEP \eqref{thm:NFSEP_real_ew} \\ \hline %\hline
(b) & Generalized     & Standard         & Standard        & Generalized         \\
    & non-Hermition   & non-Hermition    & Hermition       & Hermition \& HPD    \\ \hline
(c) & Complex         & Complex          & Real (Thm.~\ref{lem:skew-symm_A}) & Real \\ \hline
(d) & $6n\times 6n$   & $4n\times 4n$    & $6n\times 6n$   & $4n \times 4n$      \\ \hline %\hline 
(e) & $2n$ (Thm.~\ref{thm:svd_C}) & $0$ (Thm.~\ref{thm:nfep})  & $2n$  & $0$     \\ \hline %\hline
(f) & S.I. Arnoldi    & S.I. Arnoldi     & S.I. Lanczos    & S.I. Lanczos \\ \hline %\hline
(g) & Hard to choose  & Zero             & Hard to choose  & Zero         \\ \hline %\hline
(h) & \eqref{eq:shift_invert_LS} & -  
    & - & \eqref{eq:spd_LS} \\ \hline
(i) & LU or GMRES     & -           & - & CG with FFT\\
    & (not efficient) &                  &   & Mtx-Vec Mult. \\ \hline
(j) & Hard to find    & None (well-cond.)& Harder to find  & None (well-cond.) \\ \hline %\hline
(k) & Media without   & Media without    & Media without   & Media satisfying    \\ 
    & Assum.~\eqref{eq:assump} & Assum.~\eqref{eq:assump} 
    & Assum.~\eqref{eq:assump} & Assum.~\eqref{eq:assump} \\
    &                 &                  &                 & (Thm.~\ref{thm:NFSEP_real_ew}) \\ \hline
\end{tabular}
\caption{A summary of the eigenvalue problems and solvers considered
  in this article. The table lists (a) names of the eigenvalue problems, 
  (b) type of the eigenvalue problems,  (c) type of eigenvalues, 
  (d) problem dimensions, (e) number of zero eigenvalues,
  (f) eigenvalue solver, (g) choice of shift $\sigma$ for the smallest eigenvalues,
  (h) embedded linear systems, (i) linear system solvers,
  (j) preconditioners for the linear system solvers,
  and (k) the applicable complex media.}
  \label{tab:summary}
\end{table}
% ======================================================================

% ======================================================================
\begin{algorithm}
  \begin{algorithmic}[1]
    \STATE Compute $\Lambda_1$, $\Lambda_2$ and $\Lambda_3$ in Theorem
    \ref{thm:SchurDecomp_Ci}; \STATE Compute $\Lambda_{q}$ and
    $\Lambda_{p}$ in \eqref{eq:mtx_Lambda_p_q}; \STATE Compute $\Pi_1$
    and $\Pi_2$ in \eqref{eq:mtx_Pi_1} and \eqref{eq:mtx_Pi_2},
    respectively; \IF {$\mu_{d}$ and $\Phi$ are Hermitian positive
      definite and $\xi_{d}^{\ast} = \zeta_{d}$} \STATE Solve the
    HHPD-NFGEP
    \begin{align*}
      \left( \imath \begin{bmatrix} 0 & \Sigma_{r}^{-1} \\
          -\Sigma_{r}^{-1} & 0 \end{bmatrix} \right) y = \omega^{-1}
      A_{r} y,
    \end{align*}
    where $A_{r}$ is defined in \eqref{eq:factor_mtx}; \ELSE \STATE
    Solve the NFSEP
    \begin{align*}
      \mbox{diag} \left( \Sigma_{r}^{1/2}Q_{r}^{\ast},
        \Sigma_{r}^{1/2}P_{r}^{\ast} \right) B^{-1} \mbox{diag} \left(
        P_{r}\Sigma_{r}^{1/2}, Q_{r}\Sigma_{r}^{1/2} \right) y =
      \omega y.
    \end{align*}
    \STATE Update $y = \mbox{diag} \left( \Sigma_{r}^{1/2},
      \Sigma_{r}^{1/2} \right) y$;
    \ENDIF
    \STATE Compute
    \begin{align*}
      x = B^{-1} \mbox{diag} \left( P_{r}, Q_{r} \right) y.
    \end{align*}
  \end{algorithmic}
  \caption{The null space free method (NSF) for solving the GEP
    \eqref{eq:discrete_GEP_EH}}
  \label{alg:NullSpaceFree}
\end{algorithm}
% ======================================================================

% ======================================================================
\section{Numerical results} 
\label{sec:numerical}
% ======================================================================

In the numerical experiments, we consider three-dimensional reciprocal chiral materials, in which
$\zeta = -\imath \gamma$. The goals of
our numerical experiments are threefold: to verify the correctness of
the proposed algorithms and the code implementation, to compare the
performance of the proposed algorithm with an existed algorithm, and
to study the performance of the proposed method in terms of iteration
numbers and execution time. The details of the numerical experiments are
as follows.

\begin{figure}
  \begin{center}
    {\includegraphics[height=3in]{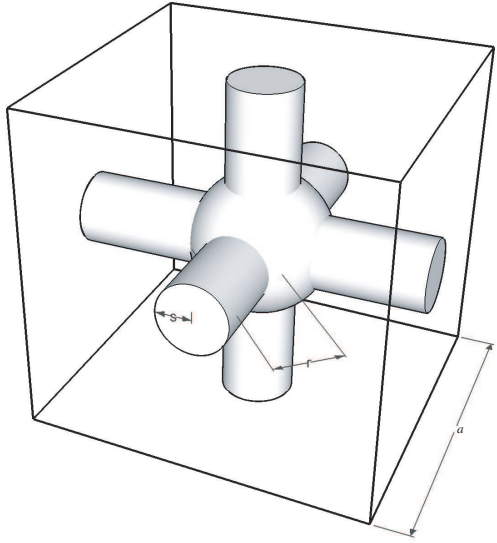}}
  \end{center}
  \caption{A schema of three-dimensional chiral medium with a simple
    cubic lattice within a single primitive cell.}
  \label{fig:schema}
\end{figure}

For the medium structure, we consider a simple cubic lattice consisting of spheres with radius $r$ and circular cylinders with radius
$s$, as shown in Figure~\ref{fig:schema}. In particular, we assume the lattice constant $a=1$, $r/a = 0.345$, and $s/a = 0.11$. We use the
triplet $(\varepsilon_{i}, \varepsilon_{o}, \gamma)$ to represent the
associated permittivity inside the structure, the permittivity outside the structure, and the
chirality parameter, respectively.  The perimeter of the irreducible Brillouin zone
for the sample cubic lattice is formed by the corners ${G} =
[0,0,0]^{\top}$, ${X} = \frac{2\pi}{a} \left[\frac{1}{2},0,
  0\right]^{\top}$, ${M} = \frac{2\pi}{a} \left[\frac{1}{2},
  \frac{1}{2},0 \right]^{\top}$, and ${R} = \frac{2\pi}{a}
\left[\frac{1}{2}, \frac{1}{2}, \frac{1}{2}\right]^{\top}$.

For the implementation, the MATLAB function {\tt eigs} is used to
solve the HHPD-NFGEP \eqref{eq:NFSEP_spd_1}, and {\tt pcg} (without
preconditioning) is used to solve the associated linear system
\eqref{eq:spd_LS}.  The stopping criteria for {\tt eigs} and {\tt pcg}
are $10^{4} \times \epsilon / (2 \sqrt{\delta_x^{-2} + \delta_y^{-2} +
  \delta_z^{-2}})$ and $10^{-14}$, respectively.  The constant
$\epsilon$ ($\approx 2.2 \times 10^{-16}$) is the floating-point
relative accuracy in MATLAB. In {\tt eigs}, the maximal number of
Lanczos vectors for the restart is $3\ell$, where $\ell =10$ is the
number of desired eigenvalues of the GEP \eqref{eq:discrete_GEP_EH}.
The MATLAB functions {\tt fftn} and {\tt ifftn} are applied to compute
the matrix-vector products $T^{\ast} \p$ and $T \q$, respectively.  The
MATLAB commands {\tt tic} and {\tt toc} are used to measure the
elapsed time. All computations are performed in MATLAB 2011b.

For the hardware configuration, we use a HP workstation 
equipped with two Intel Quad-Core Xeon X5687 3.6GHz CPUs, 48 GB of
main memory, and RedHat Linux operating system Version 5.

% ======================================================================
\subsection{Numerical correctness validation}
% ======================================================================

% ======================================================================
\begin{figure}
  \begin{center}
    \begin{tabular}{ccc}
      \subfigure[$(\varepsilon_i, \varepsilon_o,\gamma) = ( 13, 1, 0 )$ 
      \label{fig:band_structure_SC}]{\includegraphics[height=1.9in]{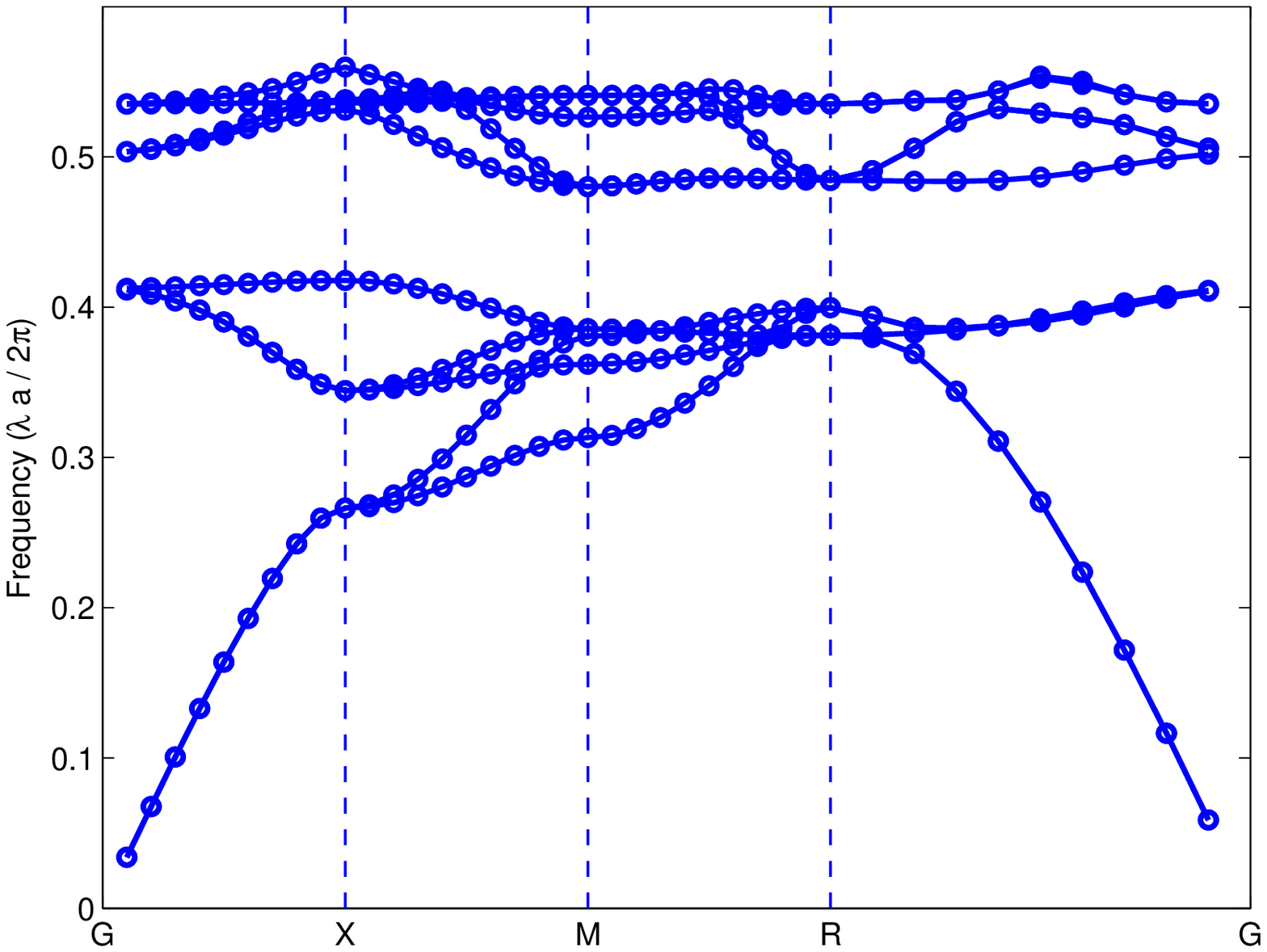}}
      &
      \subfigure[$(\varepsilon_i, \varepsilon_o,\gamma) = ( 1, 1, 0 )$ \label{fig:band_structure_Air}]{\includegraphics[height=1.9in]{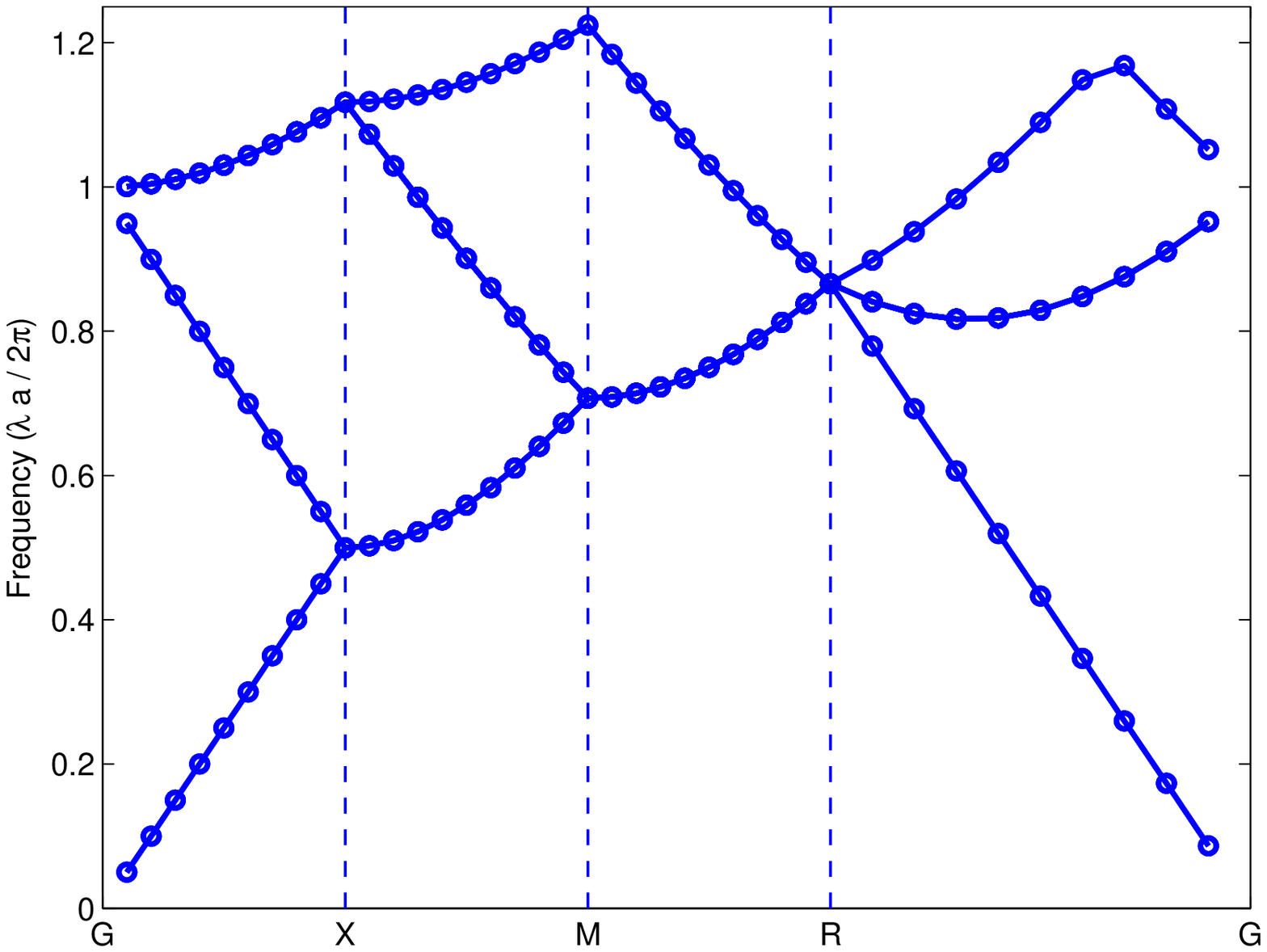}}  
    \end{tabular}
    \caption{The band structure with $(\varepsilon_i,
      \varepsilon_o,\gamma) = ( 13, 1, 0 )$ and $( 1, 1, 0 )$.}
    \label{fig:band_structure_SC_air}
  \end{center}
\end{figure}

% ======================================================================
\begin{figure}
  \begin{center}
    \begin{tabular}{ccc}
      \subfigure[$(\varepsilon_{i},\varepsilon_{o}, \gamma) = (13, 1, 0.5)$
\label{fig:diff_mesh_rp5_e13}]{\includegraphics[height=1.9in]{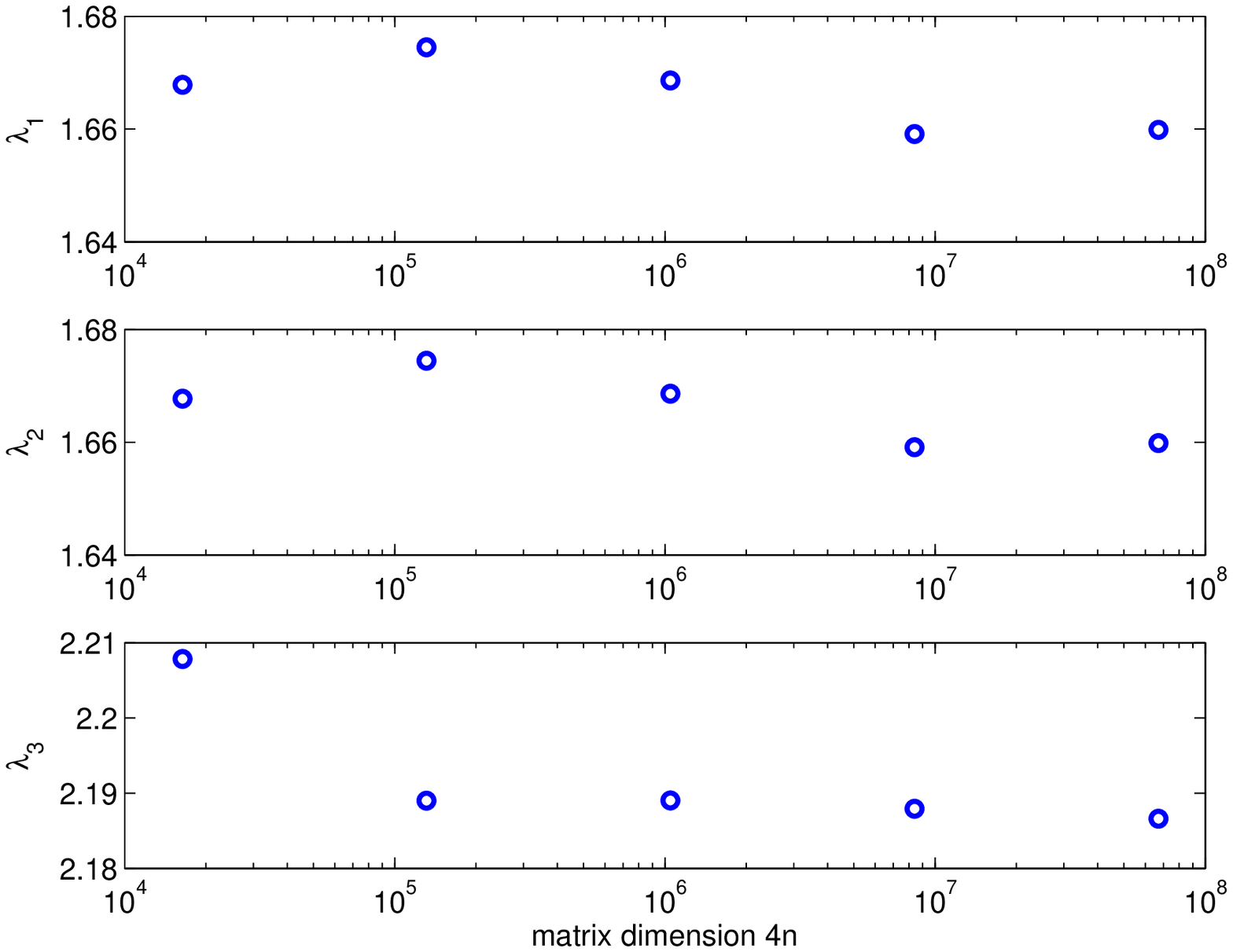}} &
\subfigure[$(\varepsilon_{i},\varepsilon_{o}, \gamma) = (1, 1, 0.8)$
\label{fig:diff_mesh_rp8_e1}]{\includegraphics[height=1.9in]{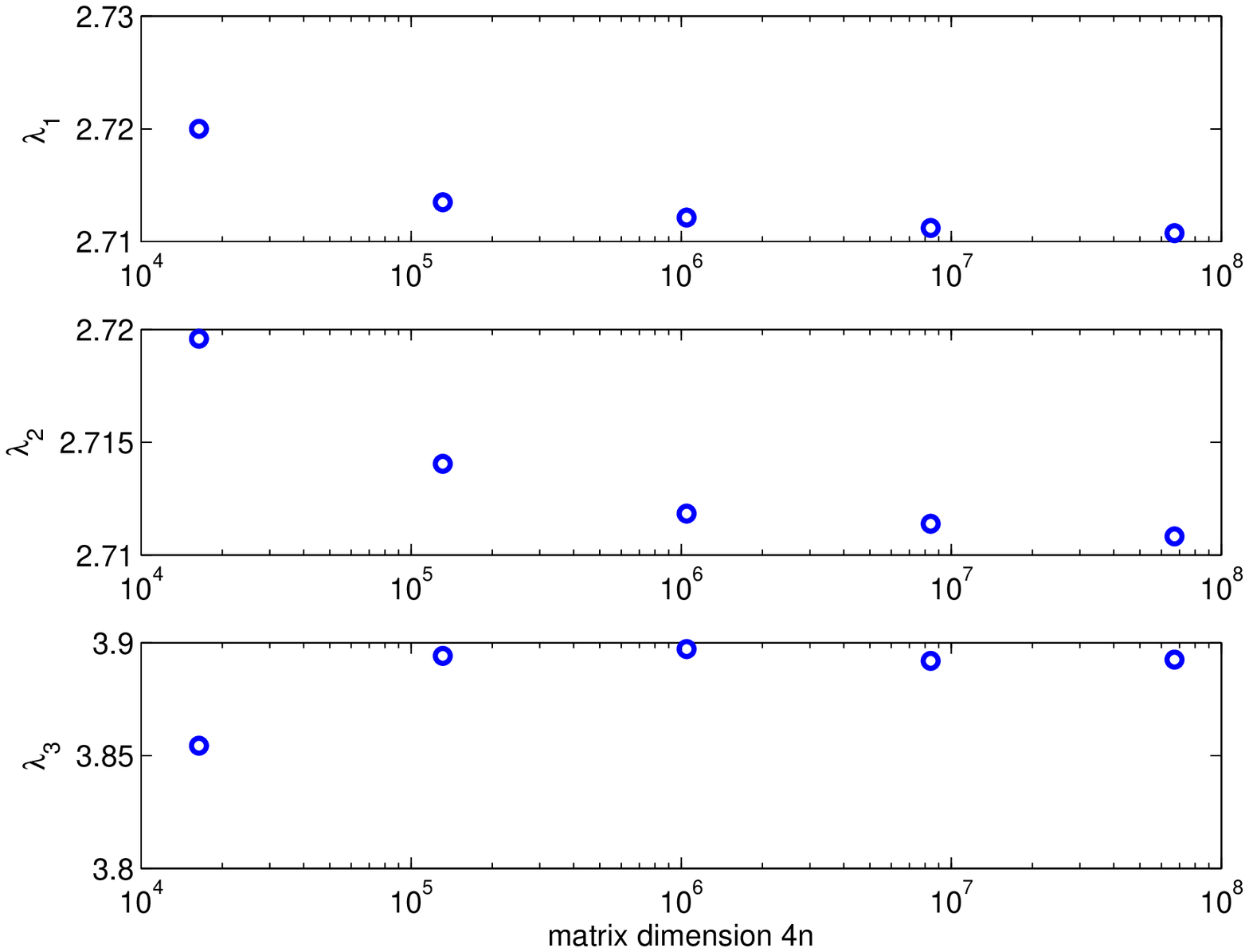}}  
    \end{tabular}
    \caption{The convergent eigenvalues for
      $(\varepsilon_{i},\varepsilon_{o}, \gamma) = (13, 1, 0.5)$ and
      $(\varepsilon_{i},\varepsilon_{o}, \gamma) = (1, 1, 0.8)$ with
      various matrix sizes $4 n$. } 
    \label{fig:diff_mesh}
  \end{center}
\end{figure}
% ======================================================================

We validate the correctness of the proposed algorithm and MATLAB
implementation by solving the following three sets of benchmark
problems.

First, we consider a special case $(\varepsilon_i, \varepsilon_o,
\gamma) = ( 13, 1, 0 )$, whose corresponding band structure has been
reported in \cite{HuangKuoWang_13}. In this case, because of
$\gamma=0$ and \eqref{eq:QEP_E}, we can see that the GEP
\eqref{eq:discrete_GEP_EH} and the eigenvalue problem $ A E = \omega^2
\varepsilon_{d} E$ (for the photonic crystal as shown in
\cite{HuangKuoWang_13}) lead to the same band structure. The
computed band structure due to \eqref{eq:discrete_GEP_EH} and $n_1 =
n_2 = n_3 = 50$ is shown in Figure~\ref{fig:band_structure_SC}. The
figure is identical (up to numerical precision) to Figure~1(a) in
\cite{HuangKuoWang_13}.

Second, we consider $(\varepsilon_i, \varepsilon_o, \gamma) = ( 1, 1,
0 )$, for which some theoretical results are known.  In this case, we know
that \eqref{eq:discrete_GEP_EH} and $ A E = \omega^2 E$ have the same
band structure and, from Theorem~\ref{thm:eigendecomp_A}, $\{
\Lambda_{q}, \Lambda_{q} \}$ are the nonzero eigenvalues of $A$. That
is, $\{ \Lambda_{q}^{1/2}, \Lambda_{q}^{1/2} \}$ are the eigenvalues of
\eqref{eq:discrete_GEP_EH}. Comparing the computed eigenvalues shown
in Figure~\ref{fig:band_structure_Air} (for $n_1 = n_2 = n_3 = 50$)
with the exact eigenvalues, our numerical results show that the
maximal relative error of all computed eigenvalues in the figure is
$3.65 \times 10^{-14}$.

Third, we check the convergence of the eigenvalues in terms of the
grid point numbers. In particular, we set $n_1 = n_2 = n_3 = 2^{k}$
for $k = 3, \ldots, 7$, and the corresponding matrix sizes $4n=4\times
2^{3k}$ of the NFSEP \eqref{eq:NFSEP_spd_1} range from $2,048$ to
$8,388,608$.  The three smallest positive eigenvalues
$\lambda_{1,k}$, $\lambda_{2,k}$ and $\lambda_{3,k}$ for
$(\varepsilon_{i}, \varepsilon_{o}, \gamma ) = (13, 1, 0.5)$ and
$(\varepsilon_{i}, \varepsilon_{o}, \gamma ) = (1, 1, 0.8)$ are shown
in Figure~\ref{fig:diff_mesh} for the wave vector $\k = [0.5, 0,
0]^{\top}$. The figure shows that $\{ \lambda_{1,k} \}$, $\{
\lambda_{2,k} \}$ and $\{ \lambda_{3,k} \}$ are convergent as $k$
increases.

% ======================================================================
\subsection{Comparison with the shift-and-invert Arnoldi method}
% ======================================================================

% ======================================================================
\begin{figure}
  \begin{center}
    {\includegraphics[height=2.5in]{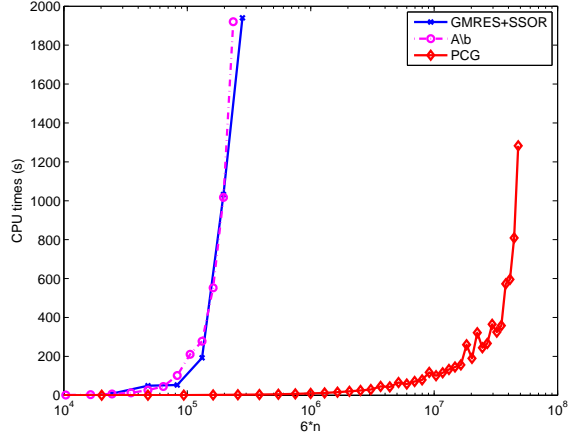}}
    \caption{The time for solving \eqref{eq:shift_invert_LS} by left
      matrix divide and the {\tt gmres} and solving \eqref{eq:spd_LS}
      by {\tt pcg}.}
    \label{fig:CPUtime_cmp_LS_solvers}
  \end{center}
\end{figure}
% ======================================================================

% ======================================================================
\begin{table}
  \centering
  \begin{tabular}{crrrrllll}
    \hline
    $\k$    & $(\frac{1}{4},0,0)$
    & $(\frac{1}{2},\frac{1}{2},0)$
    & $(\frac{1}{2},\frac{1}{2},\frac{1}{2})$ & $(\frac{1}{4},\frac{1}{4},\frac{1}{4})$ \\ \hline
    S.I. Arnoldi+LU          & $18,821$
    & $10,533$   & $16,628$ & $20,758$\\
    Algorithm~\ref{alg:NullSpaceFree}  & $155$
    & $140$    & $198$ & $155$ \\ \hline 
  \end{tabular}
  \caption{CPU time in seconds for solving the GEP 
    \eqref{eq:discrete_GEP_EH}.  S.I. Arnoldi+LU stands for the shift-and-invert 
    Arnoldi method with $LU$ based linear system solver. We take
    $n_1=n_2=n_3 = 32$ and 
    $(\varepsilon_{i}, \varepsilon_{o}, \gamma) = (13, 1, 0.5)$.}
  \label{tab:nfep_eigensolver}
\end{table}
% ======================================================================

The GEP \eqref{eq:discrete_GEP_EH} can be solved using the shift-and-invert
Arnoldi method. In the shift-and-invert Arnoldi method, the
computational cost is dominated by solving the $6n \times 6n$ linear
system
\begin{align}
  \left( \begin{bmatrix} C & 0 \\ 0 & C^{\ast}
    \end{bmatrix} - \sigma \begin{bmatrix} \zeta_{d} & I_{3n} \\
      -\varepsilon_{d} & -\xi_{d}
    \end{bmatrix} \right) y = b,
  \label{eq:shift_invert_LS}
\end{align}
for a certain vector $b$ and a shift $\sigma$. In contrast, the
performance of the null space free method
(Algorithm~\ref{alg:NullSpaceFree}) is dominated by solving the $4n
\times 4n$ linear system \eqref{eq:spd_LS}. 

We thus compare the performance for solving these two linear systems. Here, we take $(\varepsilon_{i}, \varepsilon_{o}, \gamma) =
(13, 1, 1)$ and $\k = [ 0.5, 0.5, 0]^{\top}$.  To solve
\eqref{eq:shift_invert_LS}, we use (i) a direct method based on $LU$
factorization and (ii) the GMRES with SSOR preconditioner. To solve
\eqref{eq:spd_LS}, we use the MATLAB {\tt pcg} without
preconditioning. Note that the chirality parameter $\gamma = 1$
implies that the coefficient matrix is Hermitian and positive definite.
The timing results for solving \eqref{eq:shift_invert_LS} and
\eqref{eq:spd_LS} are shown in
Figure~\ref{fig:CPUtime_cmp_LS_solvers}.  The results suggest that the
performance of the {\tt pcg} for solving \eqref{eq:spd_LS} outperforms
the two solvers for solving \eqref{eq:shift_invert_LS}
remarkably.

In Table~\ref{tab:nfep_eigensolver}, we further compare the
performance of the two eigenvalue solvers: (i) the shift-and-invert
Arnoldi method with direct linear system solver and (ii) the
null space free method (Algorithm~\ref{alg:NullSpaceFree}) with {\tt
  pcg}. It is clear that Algorithm~\ref{alg:NullSpaceFree} is much
faster.  Consequently, we do not recommend solving the GEP
\eqref{eq:discrete_GEP_EH} by the shift-and-invert Arnoldi method
unless we can develop a good preconditioning scheme for solving the
linear system \eqref{eq:shift_invert_LS}.  However, it is important to
note that, even with a good preconditioner, the effect of a large
null space (with rank $2n$) can downgrade the performance
significantly \cite{HuangHsiehLinWang_13_RePort}.

% ======================================================================
\subsection{Performance of Algorithm~\ref{alg:NullSpaceFree}}
% ======================================================================

We now concentrate on the performance, in terms of iteration numbers
and timing, of the null space free method
(Algorithm~\ref{alg:NullSpaceFree}) in finding the $10$ smallest positive
eigenvalues of the GEP \eqref{eq:discrete_GEP_EH} with various
combinations of the parameters $\varepsilon_{i}$, $\gamma$, and
$\k$. We take $n_1 = n_2 = n_3 = 128$, and the size of the coefficient
matrix in \eqref{eq:NFSEP_spd_1} is $4 n_1^3 = 8,388,608$. The Lanczos
method is applied to solve the HHPD-NFGEP \eqref{eq:NFSEP_spd_1}.

In the first test problem set, we vary the wave vector $2\pi \k$
along the segments connecting $G$, $X$, $M$, $R$, and $G$ in the first
Brillouin zone to plot the band structure. In each of the segments,
ten uniformly distributed sampling wave vectors are chosen. The
results are shown in Figure~\ref{fig:FCC_rrimitive_cell} for
$(\varepsilon_{i}, \varepsilon_{o}, \gamma) = (13,1,0.5)$ or
$(\varepsilon_{i}, \varepsilon_{o}, \gamma) = (1,1,0.8)$. 
In the second test problem set, we change the chirality parameter
$\gamma$ from $0.25$ to $3$ for $\varepsilon_{i} = 13$ (from
$\gamma=0.05$ to $0.9$ for $\varepsilon_{i} = 1$).  Note that the
condition number of the linear system \eqref{eq:spd_LS} increases from
$\frac{\varepsilon_{i}}{\varepsilon_{o}}=13$ to $\infty$ (singular) as
$\gamma$ varies from $0$ to $\sqrt{\varepsilon_{i}}\approx 3.61$. We
fix $\k = [0.5, 0, 0]^{\top}$. The results are shown in
Figure~\ref{fig:band_gamma}.  Based on
Figures~\ref{fig:FCC_rrimitive_cell} and \ref{fig:band_gamma}, we
highlight the following performance results.
\begin{itemize}
\item {\bf The iteration numbers are very small.}
  Figures~\ref{fig:iter_Krylov_eps13_k}, \ref{fig:iter_Krylov_eps1_k}
  \ref{fig:iter_Krylov_eps13_gamma}, and
  \ref{fig:iter_Krylov_eps1_gamma} show the iteration numbers of the
  Lanczos method for solving \eqref{eq:NFSEP_spd_1} with different
  parameter combinations.  In all cases, the iteration numbers are
  substantially smaller than the matrix size $8,388,608$.  Note that some
  higher iteration numbers in Figures~\ref{fig:iter_Krylov_eps13_k}
  and \ref{fig:iter_Krylov_eps1_k} are due to the corresponding
  multiplicity of eigenvalues. The higher iteration numbers in
  Figures~\ref{fig:iter_Krylov_eps13_gamma} and
  \ref{fig:iter_Krylov_eps1_gamma} are due to the clustering of the
  eigenvalues.
\item {\bf The timing is satisfactory.}  Figures~\ref{fig:CPUtime} and
  \ref{fig:CPUtime_1} show the CPU times for solving
  \eqref{eq:discrete_GEP_EH}. These results suggest that our approach
  is efficient for computing $10$ (interior) eigenvalue problems as large
  as $8,388,608$. This efficiency is mainly due to the highly
  efficient linear system solver for \eqref{eq:spd_LS}.

\item {\bf The linear systems in the form of \eqref{eq:spd_LS} are
    well-conditioned for the tested $\gamma$'s.}  We take a close look
  of the behavior of {\tt pcg} for solving the linear systems
  \eqref{eq:spd_LS} with various $\gamma$'s.  As shown in
  Figures~\ref{fig:iter_pcg_gamma_eps13} and
  \ref{fig:iter_pcg_gamma_eps1}, the (average) {\tt pcg} iteration
  numbers for solving linear system \eqref{eq:spd_LS} in the tested
  eigenvalue problems increase from $39$ to $90$ for $\varepsilon_{i}
  = 13$ and increase from $9$ to $59$ for $\varepsilon_{i} = 1$.  This
  behavior is parallel to the increase in the condition number of the
  linear system \eqref{eq:spd_LS} from
  $\frac{\varepsilon_{i}}{\varepsilon_{o}}=13$ to $\infty$ (singular)
  as $\gamma$ varies from $0$ to $\sqrt{\varepsilon_{i}}\approx
  3.61$. We fix $\k = [0.5, 0, 0]^{\top}$. A more important
  observation is that the timing results are quite satisfactory in the
  following sense.  For problems as large as $8.4$ million and a
  stopping tolerance as small as $10^{-14}$, these small iteration
  numbers suggest that the coefficient matrix in \eqref{eq:spd_LS} is
  quite well-conditioned.
\end{itemize}

% ======================================================================
\begin{figure}
  \begin{center}
    \begin{tabular}{ccc}
      \subfigure[Band structure for $(\varepsilon_{i}, \varepsilon_{o},  \gamma) = (13, 1, 0.5)$\label{fig:band_structure}]{\includegraphics[height=1.9in]{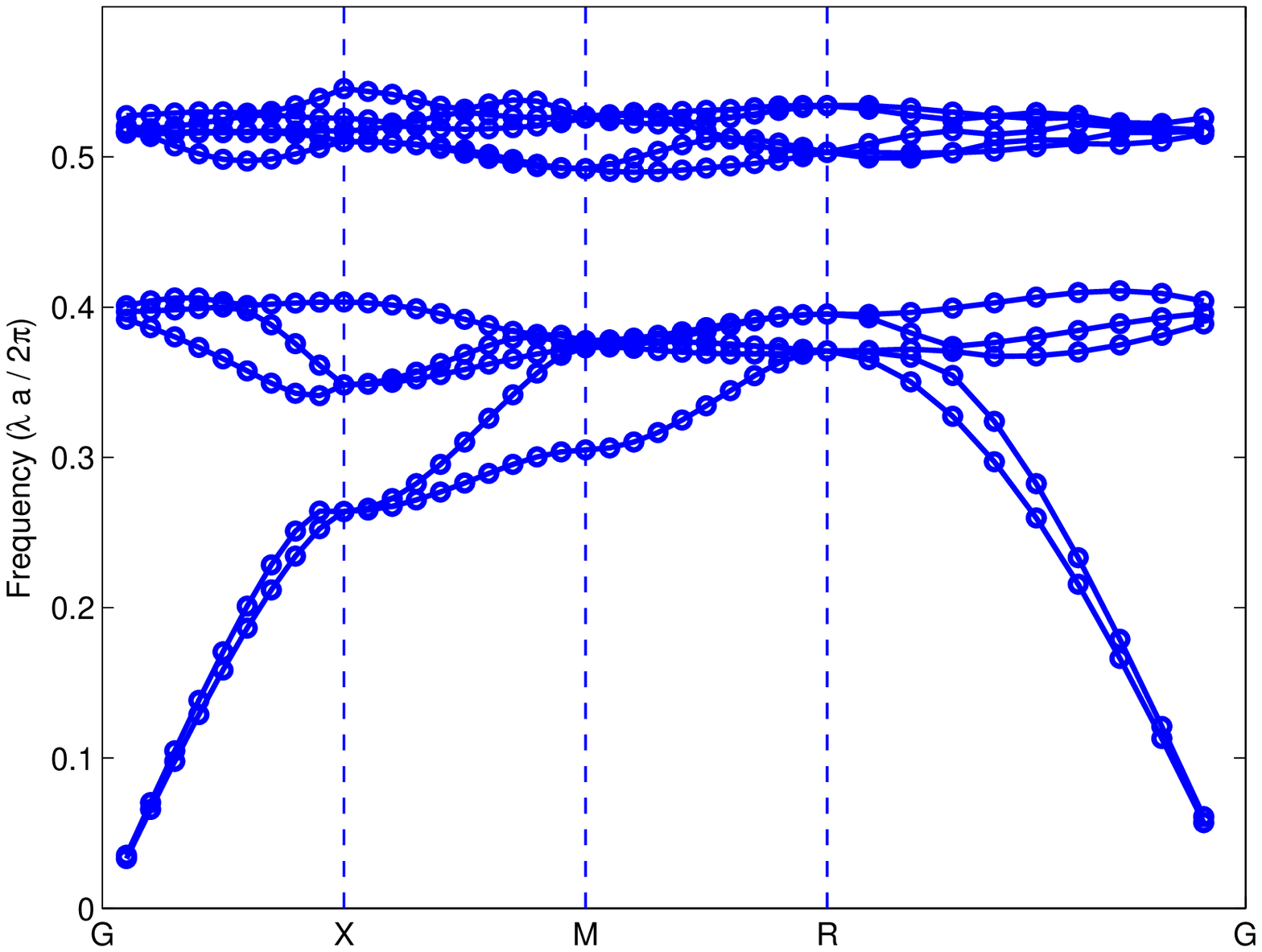}}
      & 
      \subfigure[Band structure for $(\varepsilon_{i}, \varepsilon_{o},  \gamma) = (1, 1, 0.8)$\label{fig:band_structure_1}]{\includegraphics[height=1.9in]{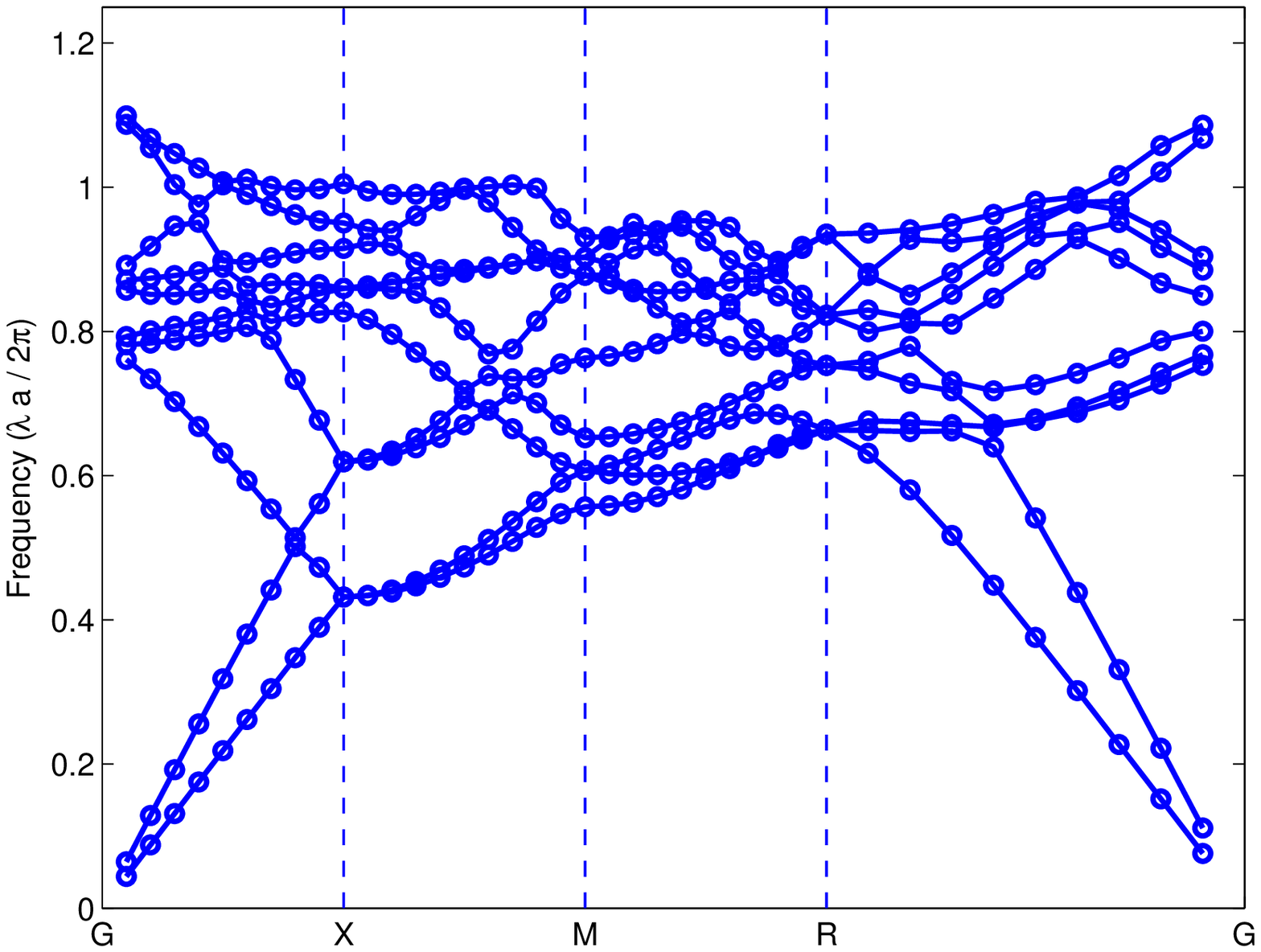}} \\
      \subfigure[Iteration numbers ranging
      from $120$ to $190$ with average $143$ for $(\varepsilon_{i}, \varepsilon_{o},  \gamma) = (13, 1, 0.5)$\label{fig:iter_Krylov_eps13_k}]{\includegraphics[height=1.9in]{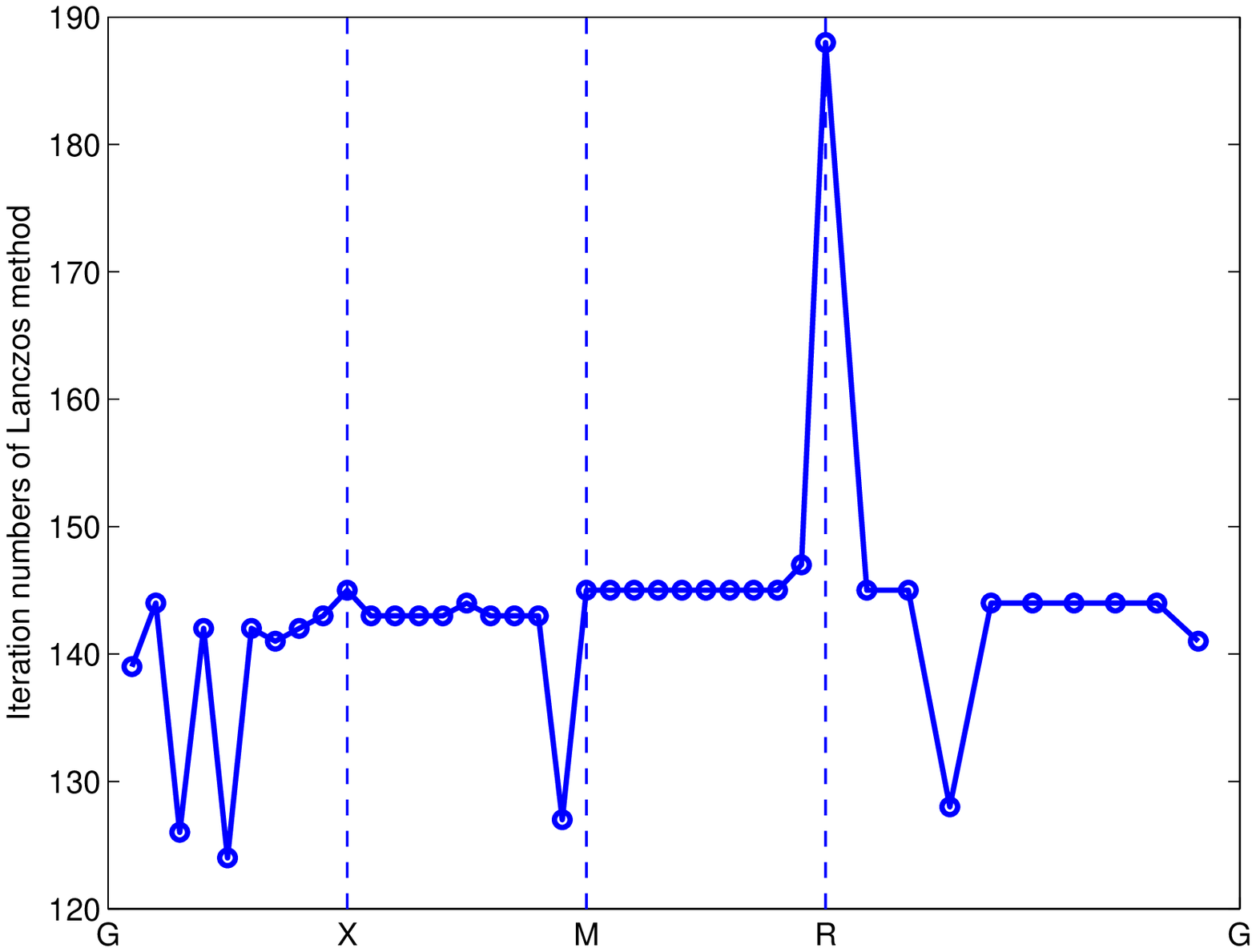}} & 
      \subfigure[Iteration numbers ranging from $100$ to $200$ with average $136$ for $(\varepsilon_{i}, \varepsilon_{o},  \gamma) = (1, 1, 0.8)$\label{fig:iter_Krylov_eps1_k}]{\includegraphics[height=1.9in]{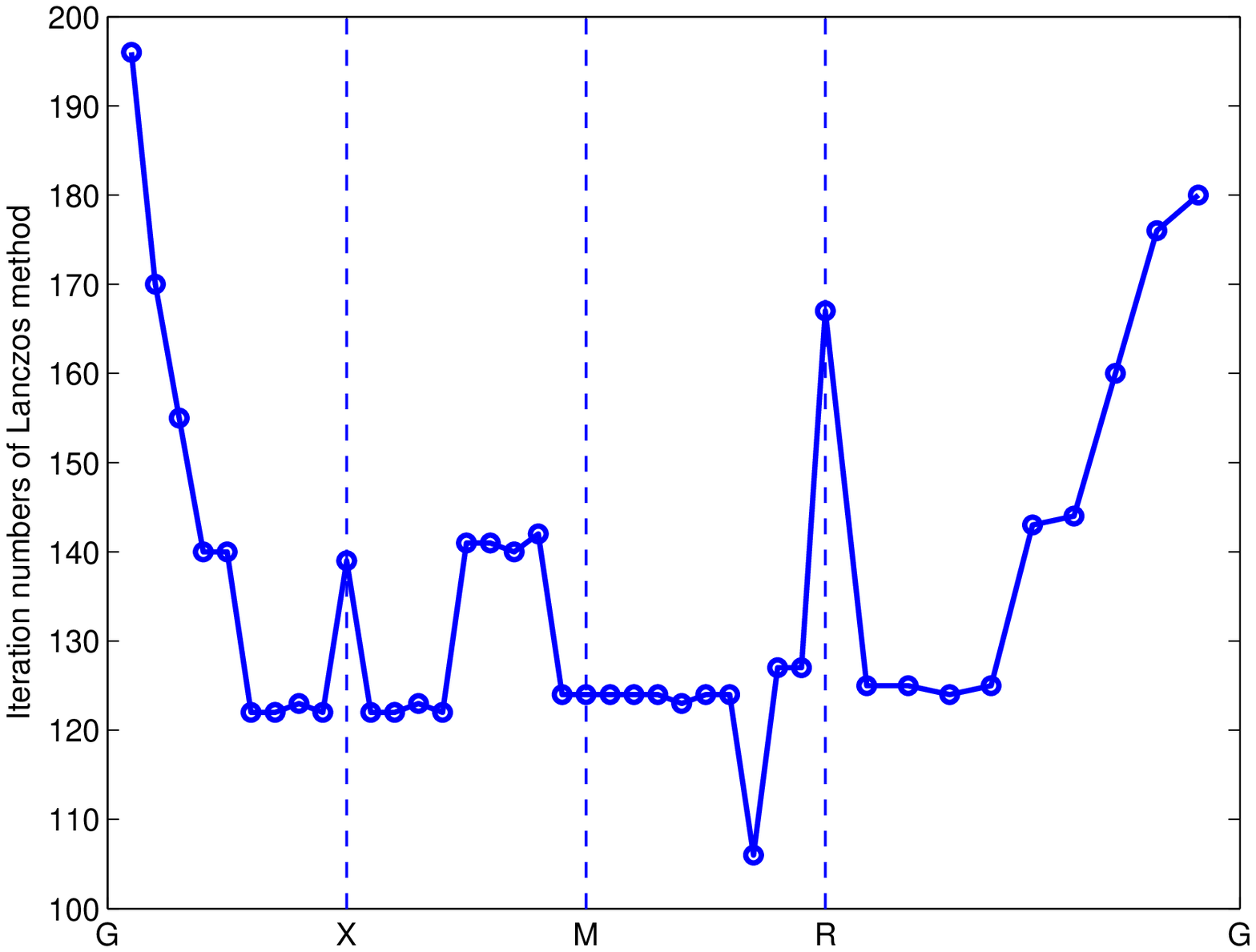}} \\
      \subfigure[CPU time ranging from $2.75$ to $4$ hours with average $3.2$ hours for $(\varepsilon_{i}, \varepsilon_{o},  \gamma) = (13, 1, 0.5)$\label{fig:CPUtime}]{\includegraphics[height=1.9in]{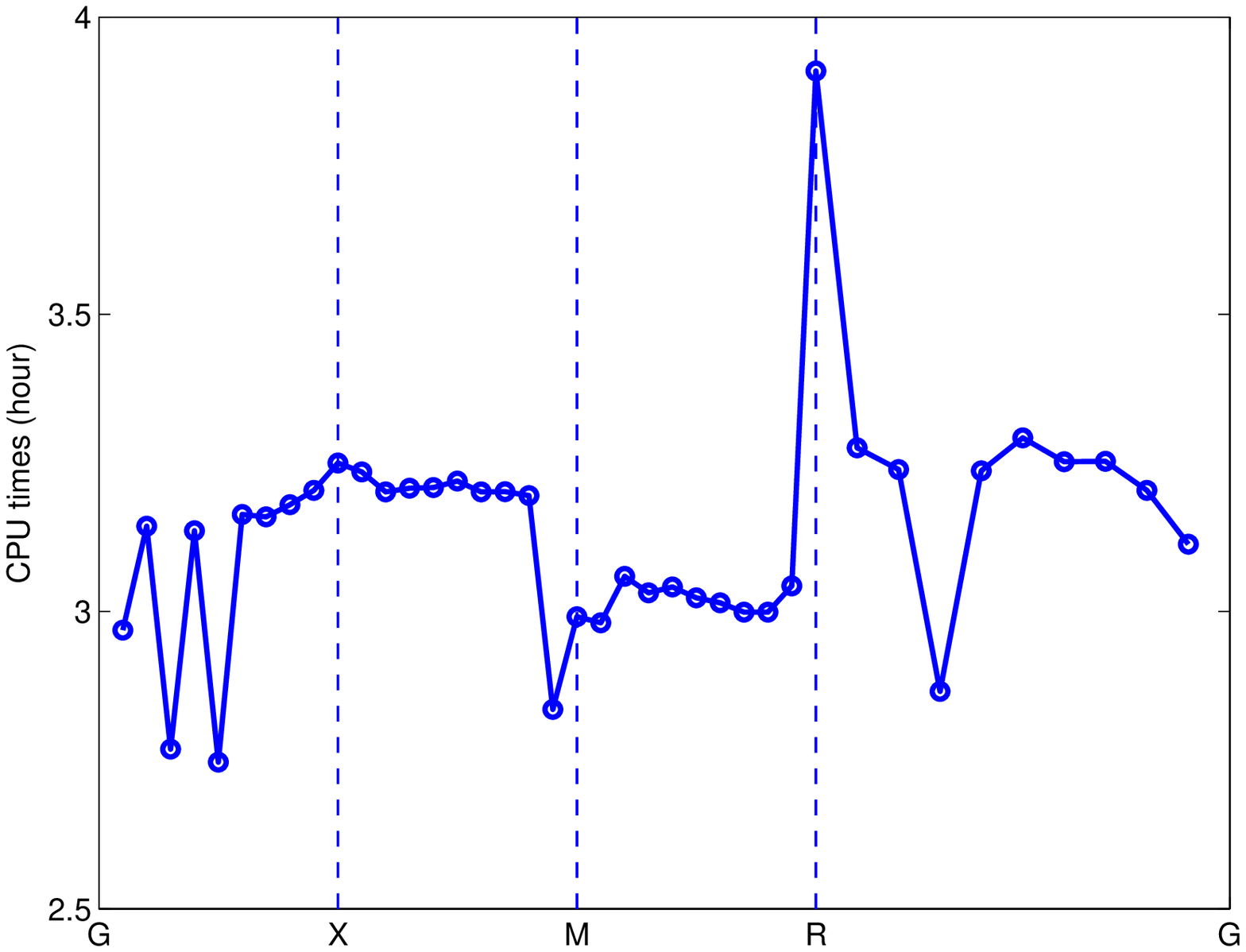}} 
      &
      \subfigure[CPU time ranging from $2.1$ to $4.1$ hours with average $2.8$ hours for $(\varepsilon_{i}, \varepsilon_{o},  \gamma) = (1, 1, 0.8)$\label{fig:CPUtime_1}]{\includegraphics[height=1.9in]{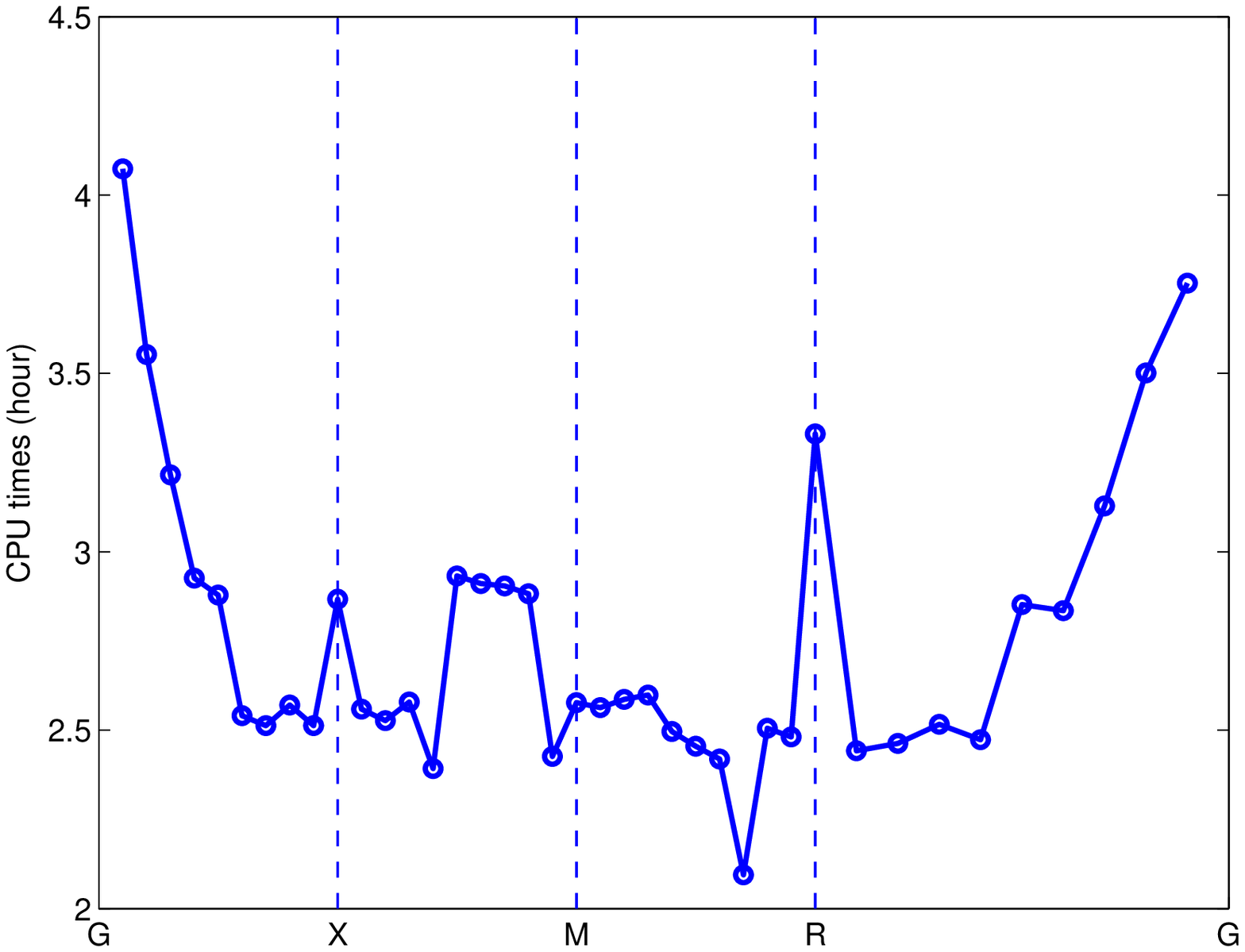}}
    \end{tabular}
    \caption{The band structures of three-dimensional chiral
      media, iteration number of the Lanczos method and the
      elapsed times of the NSF (Algorithm~\ref{alg:NullSpaceFree}) for
      solving \eqref{eq:discrete_GEP_EH} associated with various wave
      vectors $\k$.}
    \label{fig:FCC_rrimitive_cell}
  \end{center}
\end{figure}
% ======================================================================

% ======================================================================
\begin{figure}
  \begin{center}
    \begin{tabular}{ccc}
      \subfigure[Band structure for
      $(\varepsilon_{i}, \varepsilon_{o}, \gamma) = (13, 1, \gamma)$
\label{fig:band_gamma_eps13}]{\includegraphics[height=1.9in]{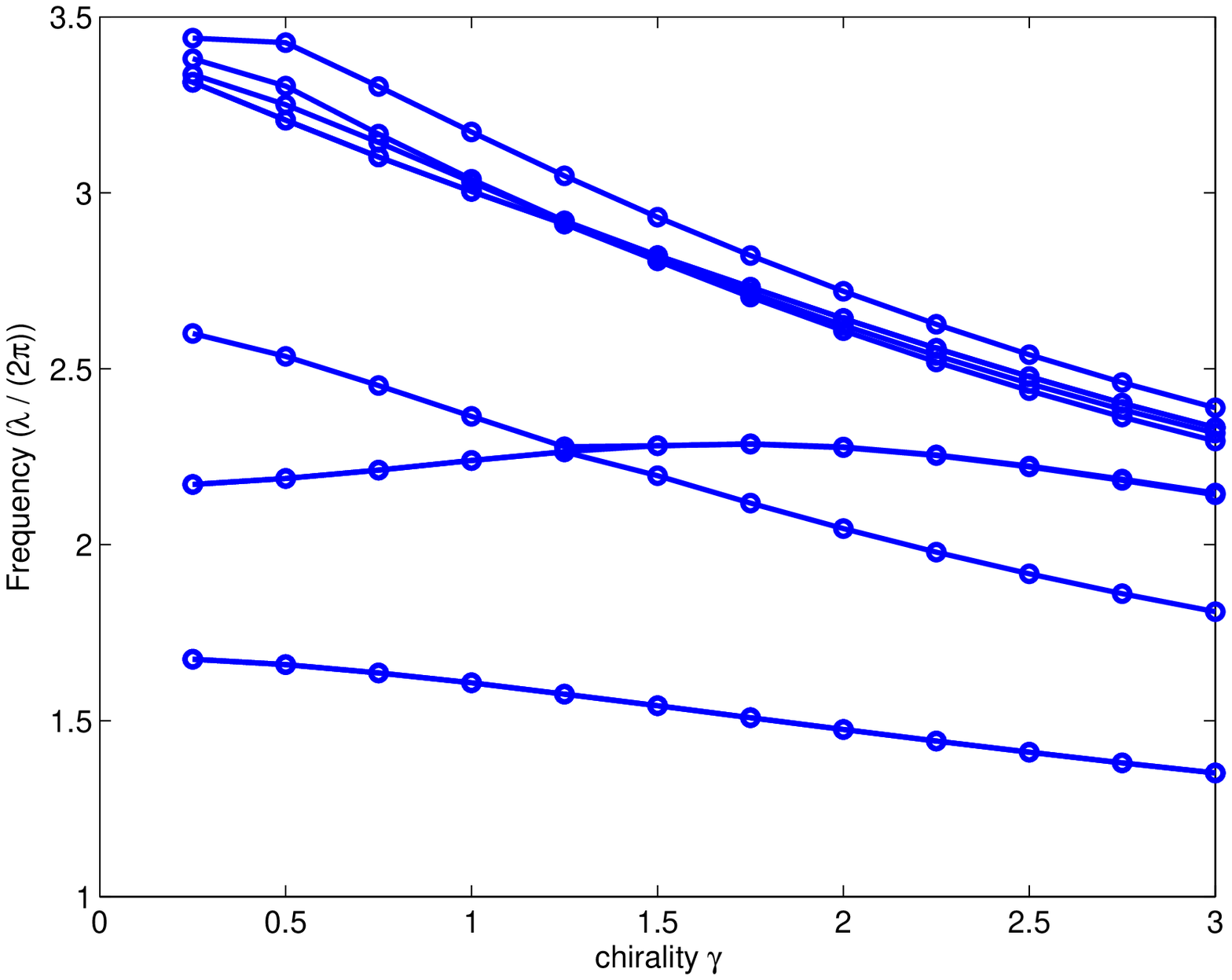}}
      & \subfigure[Band structure for 
      $(\varepsilon_{i}, \varepsilon_{o}, \gamma) = (1, 1, \gamma)$
      \label{fig:band_gamma_eps1}]{\includegraphics[height=1.9in]{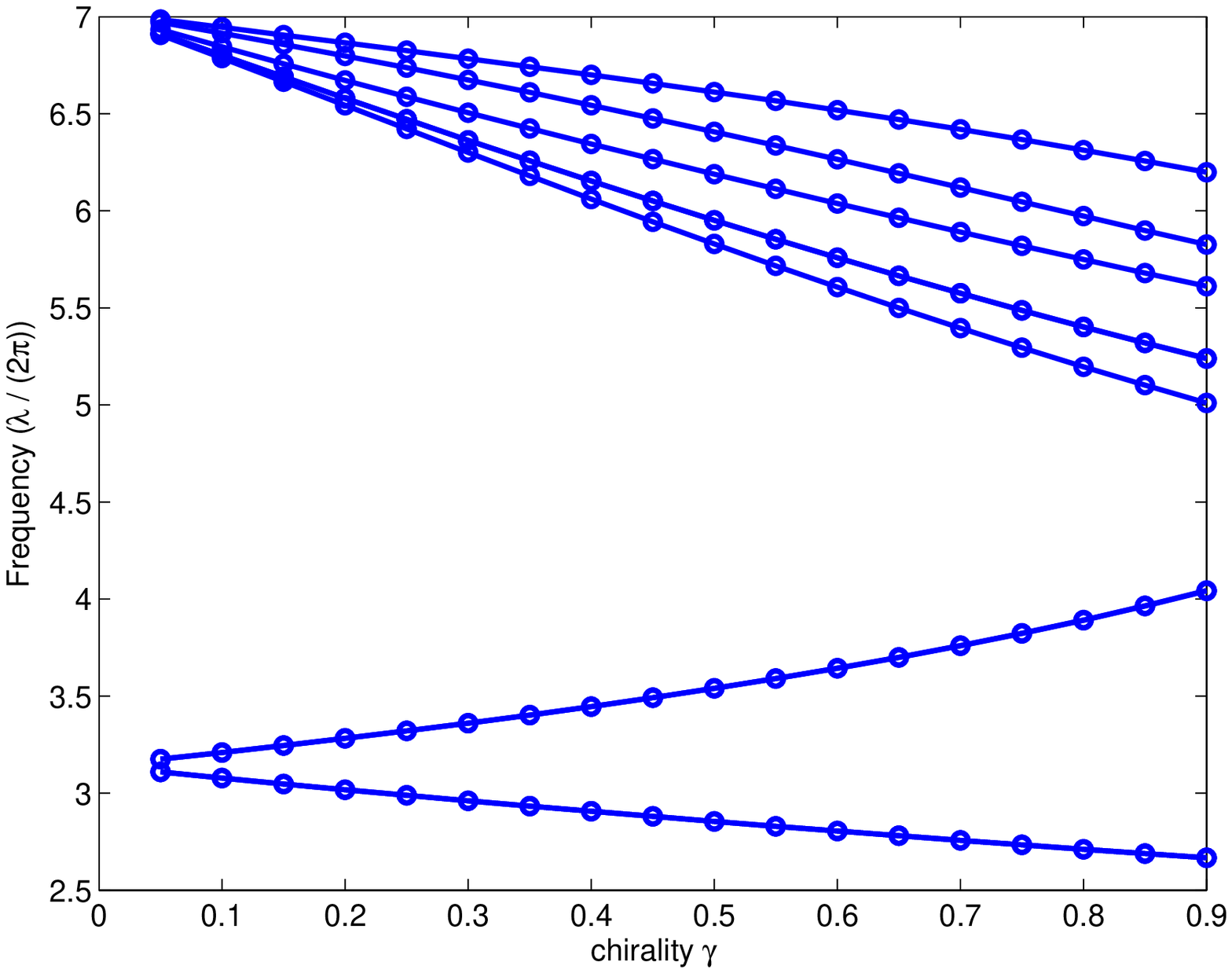}} \\
      \subfigure[Iteration numbers ranging from $124$ to $147$ for
      $(\varepsilon_{i}, \varepsilon_{o}, \gamma) = (13, 1, \gamma)$
      \label{fig:iter_Krylov_eps13_gamma}]{\includegraphics[height=1.9in]{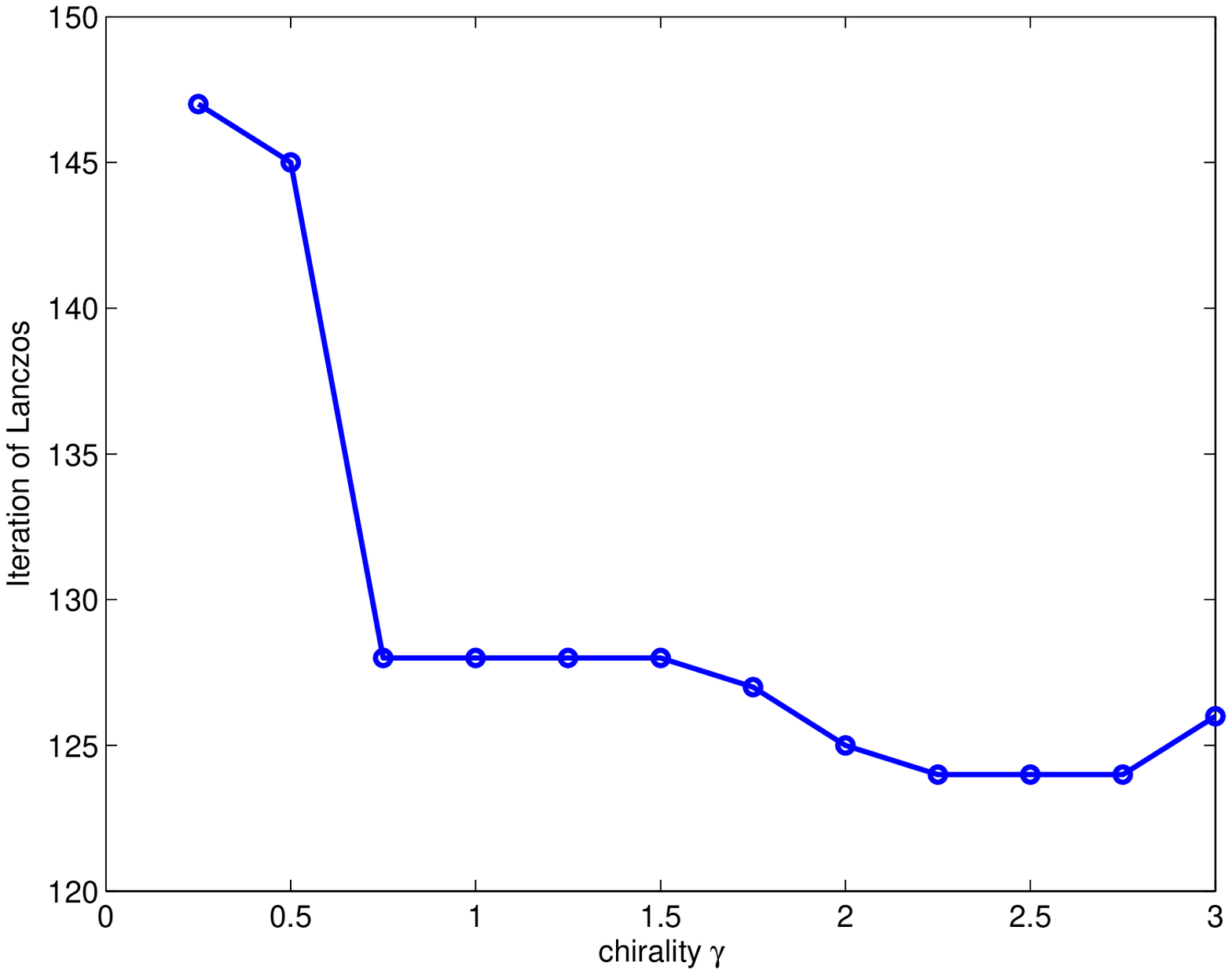}} & 
      \subfigure[Iteration numbers ranging from  $140$ to $260$ for
      $(\varepsilon_{i}, \varepsilon_{o}, \gamma) = (1, 1, \gamma)$
      \label{fig:iter_Krylov_eps1_gamma}]{\includegraphics[height=1.9in]{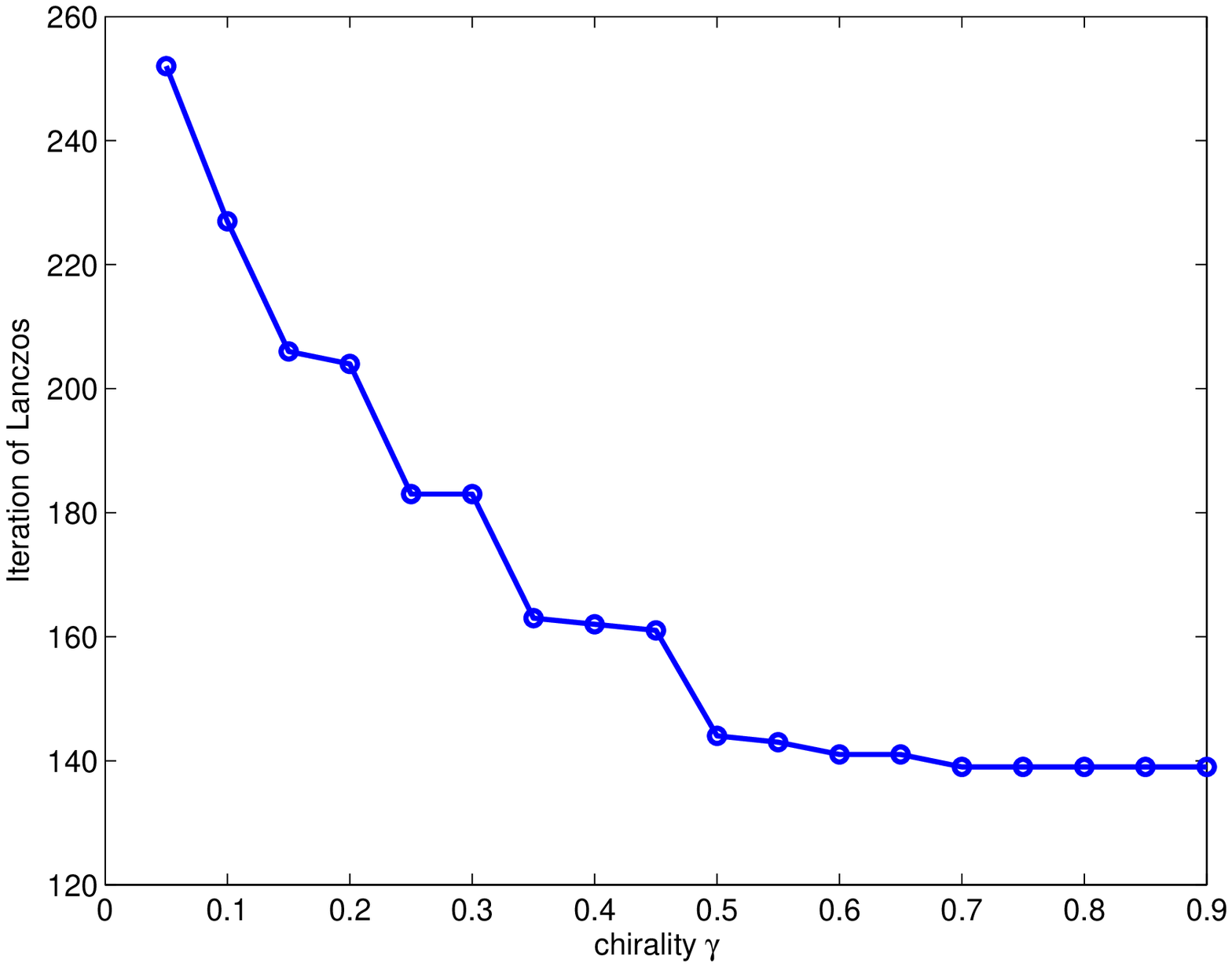}} \\
      \subfigure[CPU time increasing from $2.5$ to $6.5$ hours for 
      $(\varepsilon_{i}, \varepsilon_{o}, \gamma) = (13, 1, \gamma)$
      \label{fig:CPUtime_gamma_eps13}]{\includegraphics[height=1.9in]{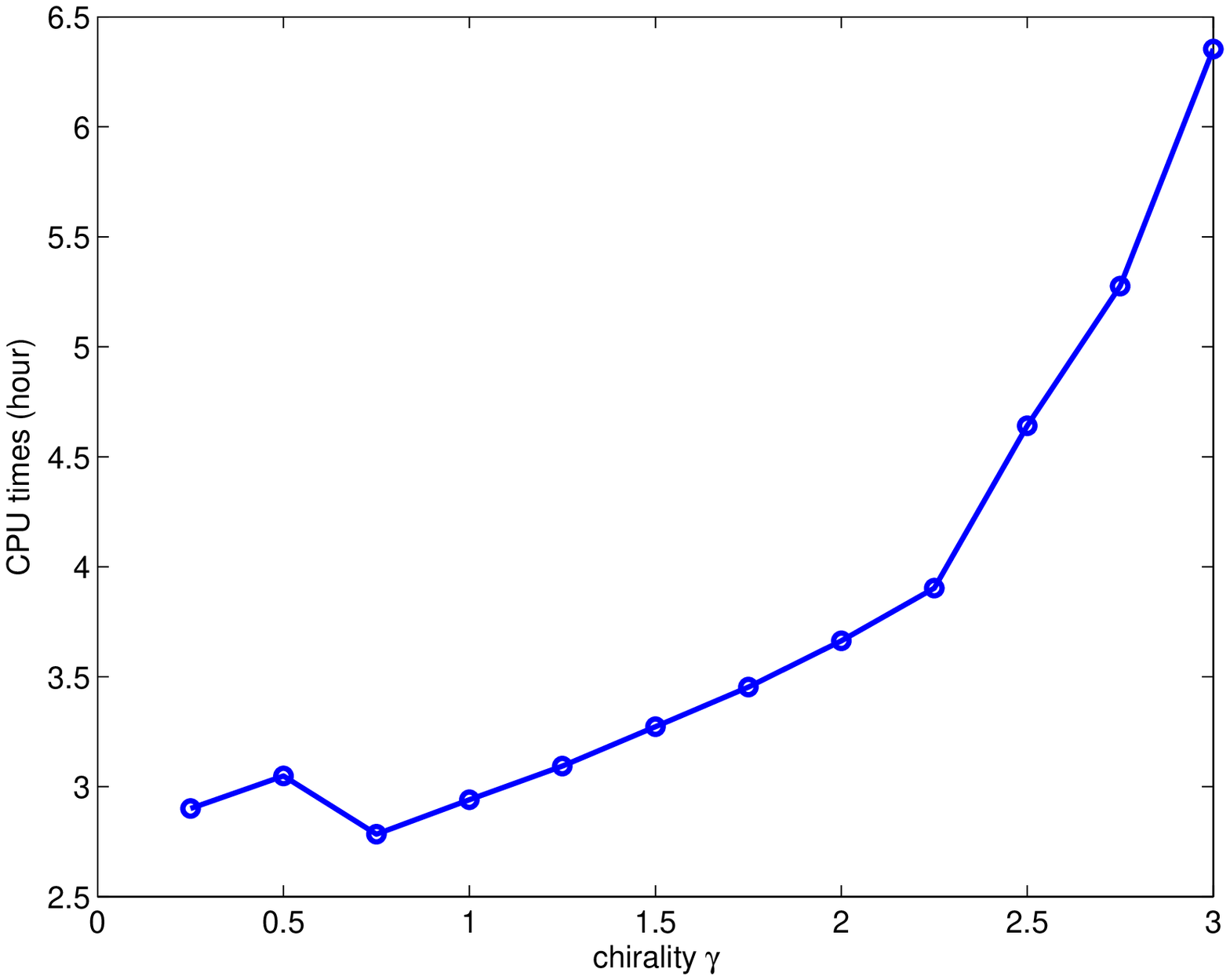}}
      &
      \subfigure[CPU time increasing from $1.4$ to $4.5$ hours for 
      $(\varepsilon_{i}, \varepsilon_{o}, \gamma) = (1, 1, \gamma)$
      \label{fig:CPUtime_gamma_eps1}]{\includegraphics[height=1.9in]{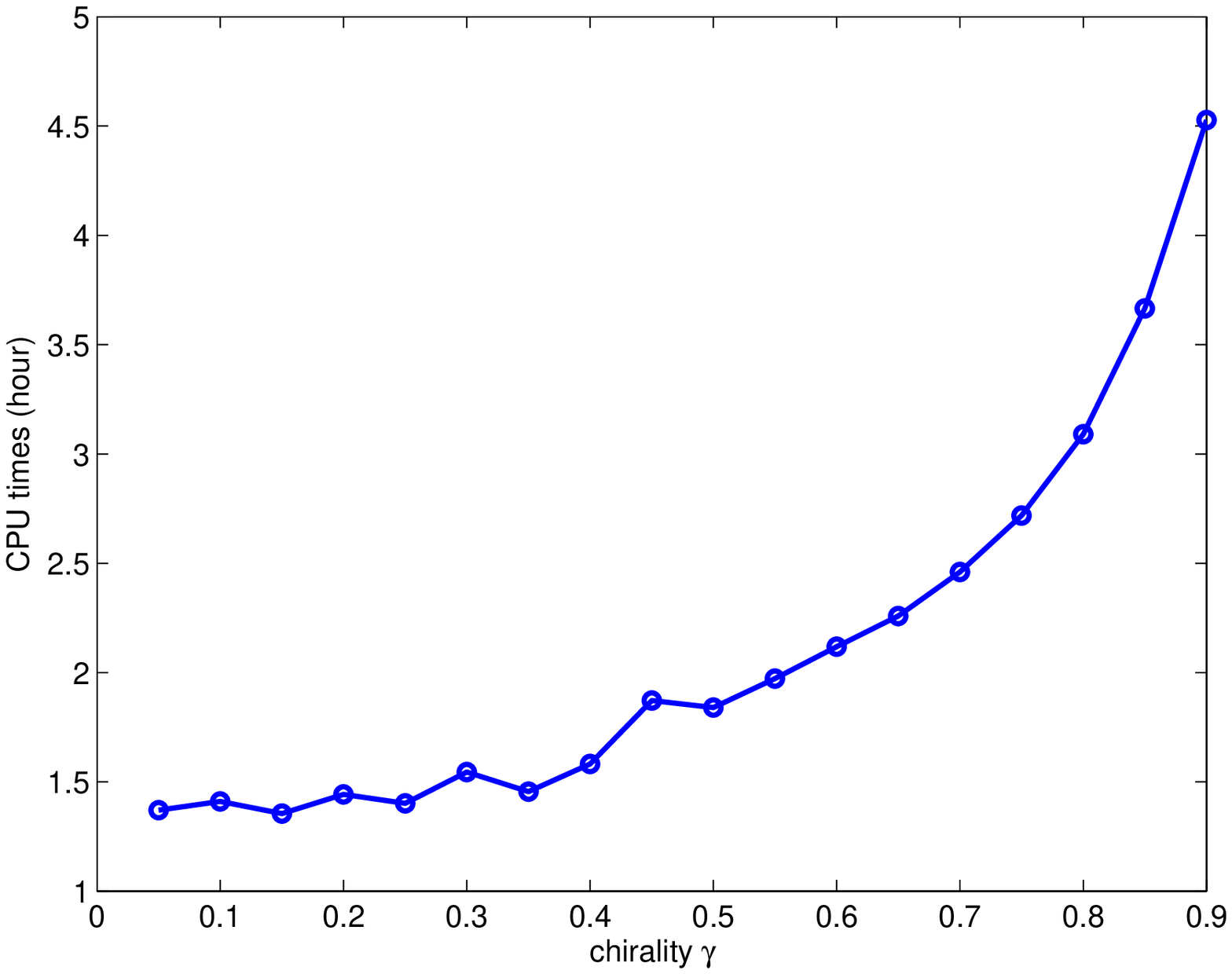}}\\
      \subfigure[Iteration numbers of CG for 
      $(\varepsilon_{i}, \varepsilon_{o}, \gamma) = (13, 1, \gamma)$
      \label{fig:iter_pcg_gamma_eps13}]{\includegraphics[height=1.9in]{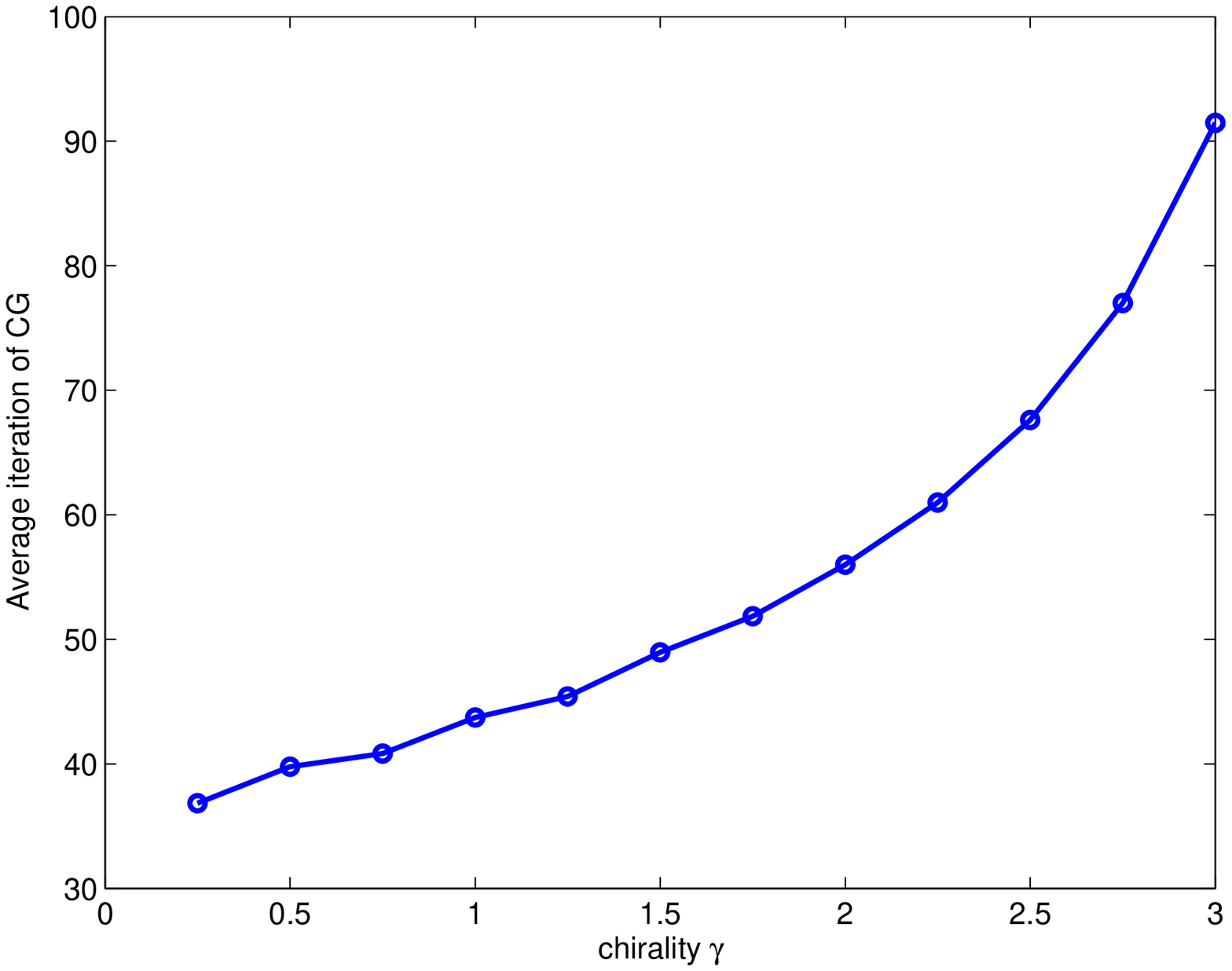}} & \subfigure[Iteration numbers of CG for 
      $(\varepsilon_{i}, \varepsilon_{o}, \gamma) = (1, 1, \gamma)$
\label{fig:iter_pcg_gamma_eps1}]{\includegraphics[height=1.9in]{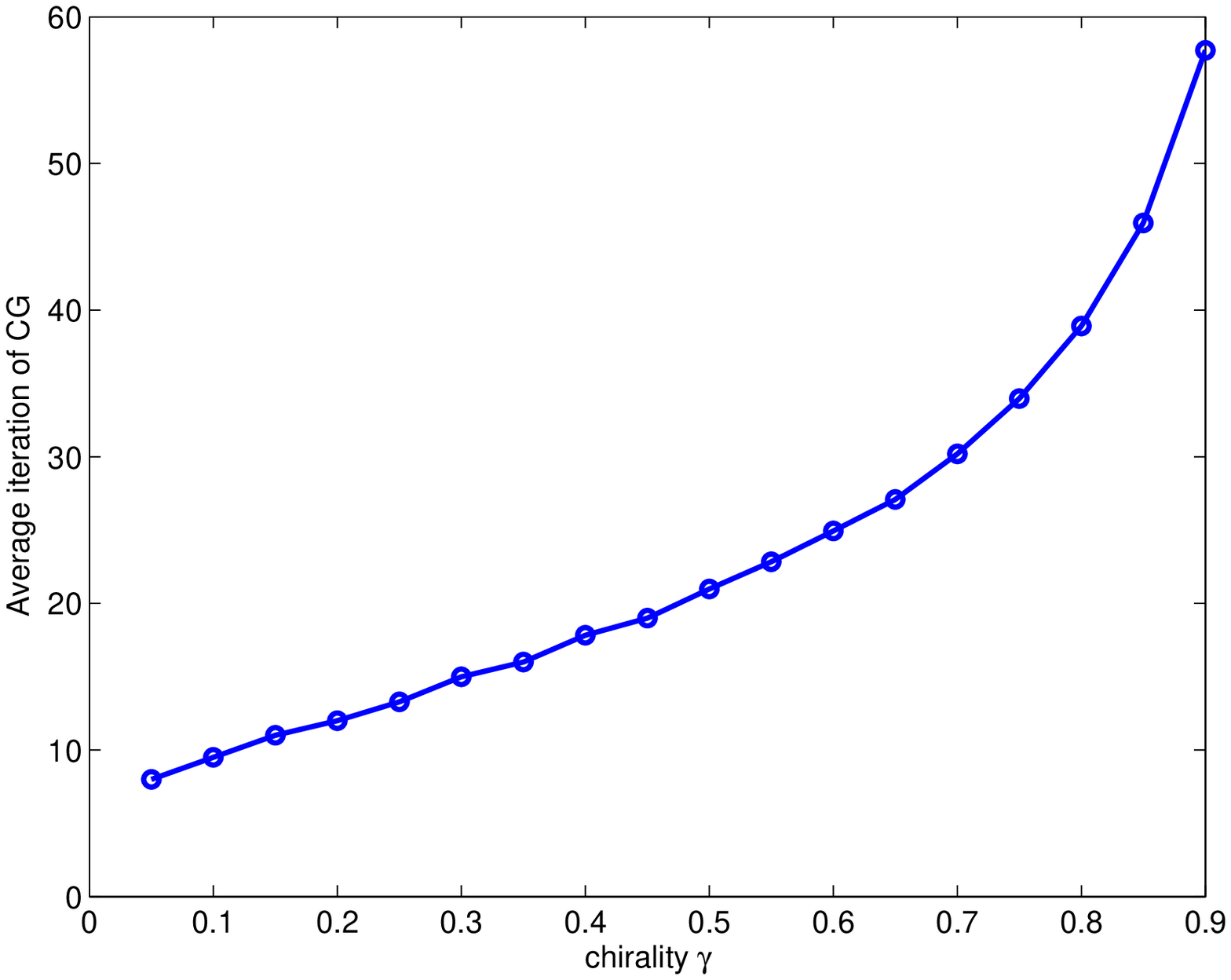}}\\
    \end{tabular}
    \caption{The band structures of three-dimensional chiral media,
      the average of iteration number of {\tt pcg}, the iteration
      numbers of the Lanczos method and the elapsed times of the NSF
      for solving \eqref{eq:discrete_GEP_EH} associated with various
      chirality parameters
      $\gamma$.} 
    \label{fig:band_gamma}
  \end{center}
\end{figure}
% ======================================================================

% ======================================================================
\section{Conclusions} 
\label{sec:conclude} 
% ======================================================================

In this paper, we focus on the generalized eigenvalue problems (GEP) arising in the
source-free Maxwell equation with magnetoelectric coupling effects in
the three-dimensional chiral media.  Solving the GEP is
a computational challenge. We have proposed a promising theoretical
framework for efficiently solving the eigenvalue problem. First, we
derive the singular value decomposition (SVD) of the discrete
single-curl operator.  Using this SVD, we explore an explicit form of
the basis for the invariant subspace corresponding to nonzero
eigenvalues of the GEP. By applying the basis, the GEP is reduced to a
null space free standard eigenvalue problem (NFSEP), which involves
only the eigenspace associated with the nonzero eigenvalues of the GEP
and excludes the zero eigenvalues so that they
do not degrade the computational efficiency.  Next, we show that all
nonzero eigenvalues of the GEP are real if $\mu_{d}$ and $\varepsilon_{d} -
\xi_{d}\mu_{d}^{-1} \zeta_{d}$ are Hermitian positive definite and
$\xi_{d}^{\ast} = \zeta_{d}$. Based on this property, we reformulate
the NFSEP to a null space free generalized eigenvalue problem whose
coefficient matrices are Hermitian and Hermitian positive definite
(HHPD-NFGEP).  We can then use the invert Lanczos method to solve the
HHPD-NFGEP and the conjugate gradient (CG) method to solve the
embedded linear systems.  The numerical results validate the correctness
of the proposed algorithms and the computer code implementation.
The results also suggest that the proposed methods are efficient in
terms of iteration and timing.

% ======================================================================
\section*{Appendix}
% ======================================================================

Proof of Lemma~\ref{lem:nonsingular_full_rank}: \\
(a). By the definition of $\Lambda_{q}$, $\Lambda_{q}$ is singular if
and only if $\Lambda_{\a_1,n_1}$, $\Lambda_{\a_2,n_2}$ and
$\Lambda_{\a_3,n_3}$ are singular if and only if
\begin{align*}
  e^{\theta_{n_1,i} + \theta_{\a_1,i}} - 1 &= 0, \\
  e^{\theta_{n_2,j} + \theta_{\a_2,j}} - 1 &= 0, \\
  e^{\theta_{n_3,k} + \theta_{\a_3,k}} - 1 &= 0
\end{align*}
for some $i \in \{ 1, \ldots, n_1\}$, $j \in \{ 1, \ldots, n_2\}$ and
$k \in \{ 1, \ldots, n_3\}$. That is
\begin{align*}
  \frac{i+\k \cdot \a_1}{n_1} = \frac{i+a \mathrm{k}_1}{n_1}, \
  \frac{j+\k \cdot \a_2}{n_2} = \frac{j+a \mathrm{k}_2}{n_2},
  \frac{k+\k \cdot \a_3}{n_3} = \frac{k+a \mathrm{k}_3}{n_3}
\end{align*}
are integers for some $i, j, k$. By the assumption $0 \leq
\mathrm{k}_1, \mathrm{k}_2, \mathrm{k}_3 \leq \frac{1}{2a}$, we have
$i = n_1$, $j = n_2$, $k = n_3$ and $\mathrm{k}_1 = \mathrm{k}_2 =
\mathrm{k}_3 = 0$ which contradict to $\k \neq 0$. Therefore,
$\Lambda_{q}$ is nonsingular.

(b). By the definitions of $\Lambda_1$, $\Lambda_2$ and $\Lambda_3$ in
\eqref{eq:eigendecomp_Cis}, the $((k-1)n_1n_2+(j-1)n_1+i)$th elements
of $\Lambda_1$, $\Lambda_2$ and $\Lambda_3$ are
\begin{align*}
  \delta_{x}^{-1} (e^{\theta_i} - 1), \quad \delta_{y}^{-1}
  (e^{\theta_j} - 1), \quad \delta_{z}^{-1} (e^{\theta_k} - 1),
\end{align*}
respectively, where
\begin{align*}
  \theta_i = \imath 2 \pi \left( \frac{i+\mathrm{k}_1}{n_1}\right),
  \quad \theta_j = \imath 2 \pi \left(
    \frac{j+\mathrm{k}_2}{n_2}\right), \quad \theta_k = \imath 2 \pi
  \left( \frac{k+\mathrm{k}_3}{n_3} \right),
\end{align*}
for $i = 1, \ldots, n_1$, $j = 1, \ldots, n_2$ and $k = 1, \ldots,
n_3$. Assume that $\Lambda_{p}$ does not have full column rank. Then
there exists some $i$, $j$ and $k$ such that
\begin{align*}
  \beta \delta_{z}^{-1} (e^{\theta_k} - 1) &=  \delta_{y}^{-1}(e^{\theta_j} - 1), \\
  \delta_{x}^{-1} (e^{\theta_i} - 1) &=  \alpha \delta_{z}^{-1} (e^{\theta_k} - 1), \\
  \beta \delta_{x}^{-1} (e^{\theta_i} - 1) &= \alpha \delta_{y}^{-1}
  (e^{\theta_j} - 1),
\end{align*}
which implies that
\begin{align*}
  \frac{\beta \sin \theta_k}{\delta_z} = \frac{ \sin
    \theta_j}{\delta_y}, \quad \frac{\sin \theta_i}{\delta_x} =
  \frac{\alpha \sin \theta_k}{\delta_z}, \quad \frac{\beta \sin
    \theta_i}{\delta_x} = \frac{\alpha \sin \theta_j}{\delta_y}
\end{align*}
and
\begin{align*}
  \frac{\beta (\cos \theta_k -1)}{\delta_z} &= \frac{\cos \theta_j-1}{\delta_y}, \\
  \frac{\cos \theta_i -1}{\delta_x} &= \frac{\alpha (\cos \theta_k-1)}{\delta_z}, \\
  \frac{\beta (\cos \theta_i-1)}{\delta_x} &= \frac{\alpha (\cos
    \theta_j-1)}{\delta_y}
\end{align*}
or equivalent to
\begin{align}
  \frac{\beta \sin \theta_i}{\alpha \delta_x} = \frac{\sin
    \theta_j}{\delta_y} = \frac{\beta \sin
    \theta_k}{\delta_z} \label{eq:sin_1}
\end{align}
and
\begin{align}
  \frac{\beta (\cos \theta_i-1)}{\alpha \delta_x} = \frac{\cos
    \theta_j - 1}{\delta_y} = \frac{\beta (\cos \theta_k-1)}{
    \delta_z}. \label{eq:cos_1}
\end{align}
From \eqref{eq:sin_1} and \eqref{eq:cos_1}, it holds that
\begin{align*}
  \left( \frac{\beta \sin \theta_i}{\alpha \delta_x}\right)^2 + \left(\frac{\beta (\cos \theta_i-1)}{\alpha \delta_x} \right)^2 &= \left( \frac{\sin \theta_j}{\delta_y} \right)^2 + \left( \frac{\cos \theta_j - 1}{\delta_y}\right)^2 \\
  &= \left(\frac{\beta \sin \theta_k}{ \delta_z} \right)^2 + \left(
    \frac{\beta (\cos \theta_k-1)}{\delta_z} \right)^2
\end{align*}
and then
\begin{align*}
  \frac{\beta^2(2 - 2\cos \theta_i)}{\alpha^2 \delta_x^2} =
  \frac{2-2\cos \theta_j}{\delta_y^2} = \frac{\beta^2(2-2\cos
    \theta_k)}{ \delta_z^2}.
\end{align*}
Therefore,
\begin{align}
  \frac{\beta (\cos \theta_i -1)}{\alpha \delta_x} = \frac{\alpha
    \delta_x}{\beta \delta_y} \frac{\cos \theta_j -1}{\delta_y}, \quad
  \frac{\beta (\cos \theta_k-1)}{\delta_z} = \frac{\delta_z}{\beta
    \delta_y} \frac{\cos \theta_j -1}{\delta_y}. \label{eq:cos_2}
\end{align}
From \eqref{eq:cos_1} and \eqref{eq:cos_2}, we can see that if $\alpha
\delta_x \neq \beta \delta_y$ and $ \delta_z \neq \beta \delta_y$,
then
\begin{align*}
  \cos \theta_i = \cos \theta_j = \cos \theta_k = 1.
\end{align*}
That is $\frac{i+\mathrm{k}_1}{n_1}$, $\frac{j+\mathrm{k}_2}{n_2}$ and
$\frac{k + \mathrm{k}_3}{n_3}$ must be integers. This contradicts to
the assumption for $\k$. Therefore, $\Lambda_{p}$ has full column
rank.

% ======================================================================
\section*{Acknowledgements}
% ======================================================================
This work is partially supported by the National Science Council, the
National Center for Theoretical Sciences, the Taida Institute for
Mathematical Sciences, and the Chiao-Da ST Yau Center in Taiwan.

\bibliographystyle{plain} 
\bibliography{research_papers}

\end{document}